\documentclass[10pt,a4paper]{article}

\usepackage{amssymb,amsmath,amsthm}	% Mathematische Symbole
\usepackage{a4wide}
\usepackage{graphicx}      										% Grafiken einbinden
\graphicspath{{figs/results/Aggregation/}{figs/results/Classification/}{figs/results/FP/}{figs/results/Regression/}{figs/results/Sensitivity/}}
\usepackage{tikz}          										% LaTeX-Grafiken erzeugen
\usepackage{url}
\usepackage{dsfont}
\usepackage{amsmath}
\usepackage{xspace}
\usepackage{framed}
\usepackage{nicefrac}
\usepackage{cite}
\usepackage[textsize=footnotesize,backgroundcolor=yellow!70,bordercolor=orange]{todonotes}

\usepackage{algorithm}
\usepackage{algorithmic}

\renewcommand{\d}{\mathrm{d}}

\newtheorem{definition}{Definition}[]

\newtheorem{theorem}{Theorem}[]
\newtheorem{remark}{Remark}[]
\newtheorem{corollary}{Corollary}[]

\newtheorem{proposition}{Proposition}[]

\newcommand{\R}{\ensuremath{\mathbb{R}}}

\begin{document}
	
	%-- TITEL ----------------------------------------------------------%
	\title{Mean-Field and Kinetic Descriptions of Neural Differential Equations}
	\author{Michael Herty\thanks{Institut f\"{u}r Geometrie und Praktische Mathematik (IGPM) -- RWTH Aachen University -- Templergraben 55, 52062 Aachen (Germany) -- \texttt{herty@igpm.rwth-aachen.de}} ,\ Torsten Trimborn\thanks{NRW.BANK -- Kavalleriestra\ss e 22, 40213 D\"{u}sseldorf (Germany) -- \texttt{torsten.trimborn@nrwbank.de}} ,\ Giuseppe Visconti\footnotemark[1]~\thanks{Currently at: Department of Mathematics ``G.~Castelnuovo'' -- Sapienza University of Rome -- P.le Aldo Moro 5, 00185 Roma (Italy) -- \texttt{giuseppe.visconti@uniroma1.it}}
	}
	\maketitle

   %-- INHALTSVERZEICHNIS ----------------------------------------------------------%

\begin{abstract} 
Nowadays, neural networks are widely used in many applications as artificial intelligence models for learning tasks. Since typically neural networks process a very large amount of data, it is convenient to formulate them within the mean-field and kinetic theory. In this work we focus on a particular class of neural networks, i.e.~the residual neural networks, assuming that each layer is characterized by the same number of neurons $N$, which is fixed by the dimension of the data. This assumption allows to interpret the residual neural network as a time-discretized ordinary differential equation, in analogy with neural differential equations. The mean-field description is then obtained in the limit of infinitely many input data. This leads to a Vlasov-type partial differential equation which describes the evolution of the distribution of the input data. We analyze steady states and sensitivity with respect to the parameters of the network, namely the weights and the bias. In the simple setting of a linear activation function and one-dimensional input data, the study of the moments provides insights on the choice of the parameters of the network. Furthermore, a modification of the microscopic dynamics, inspired by stochastic residual neural networks, leads to a Fokker-Planck formulation of the network, in which the concept of network training is replaced by the task of fitting distributions. The performed analysis is validated by artificial numerical simulations. In particular, results on classification and regression problems are presented.
\end{abstract}



\section{Introduction}
The use of machine learning algorithms has gained a lot of interest in the past decade~\cite{jordan2015machine, joshi2019machine, MuellerVincentBostrom2016, Wooldridge2020}. Besides data science problems such as clustering, regression, image recognition or pattern formation, there are novel applications in the field of engineering as e.g.~for production processes~\cite{onar2018changing, schmitt2018advances, tercan2017improving, Bobzin2021}.
In this study we focus on deep residual neural networks (ResNets) which date back to the 1970s and have been heavily influenced by the pioneering work of Werbos~\cite{werbos1994roots}. ResNets have been successfully applied to a variety of applications such as image recognition~\cite{wu2019wider}, robotics~\cite{zeng2018robotic} or classification tasks~\cite{fawaz2018data}. More recently, also applications to mathematical problems in numerical analysis~\cite{Mishra2019,HesthavenRay2018,HesthavenRay2019,WangHesthavenRay,ZhangGuoKarniadakis2020} and optimal control~\cite{Osher2019} have been considered. These contributions are usually aimed to use machine learning in a collaborative fashion.

The process of a ResNet can be shortly summarized as follows. Given input data, which are also referred to as measurements, the ResNet propagates those through layers and neurons to a final state using a precise dynamics which depends on some parameters, the weights and the bias. This final state is usually called output or prediction of the network, which hopefully solves a given learning task. The success of the ResNet depends on the choice of the parameters which have to be determined in an optimal fashion, namely by solving an optimization procedure, called training. This procedure is performed on a reference data set, characterized by a given target data for any input measurement. Then, the training aims to compute optimal weights and bias by minimizing a suitable distance between the predictions of the network and the given targets.
For the training of ResNets back-propagation algorithms based on the stochastic gradient descent method~\cite{werbos1994roots}, or the ensemble Kalman filter~\cite{Kovachki2019,Watanabe1990,Alper}, or other approaches~\cite{Ha20202417} are frequently used.

Since neural networks are typically processed on a very large amount of data, the purpose of this work is to formulate a particular class of ResNets within the mean-field and the kinetic theory in order to provide a synthetic and fast interpretation of the learning process. In fact, in this way, the dynamics of the ResNet are reinterpreted in terms of partial differential equations (PDEs) describing the evolution of the distribution of the data. In other words, we look at the data as a whole rather than considering them as single entities. There have been made several attempts to describe neural networks by differential equations~\cite{chen2018neural,lu2017beyond,ruthotto2018deep}. More precisely, in~\cite{ruthotto2018deep} the connection between deep convolution neural networks and PDEs is studied. In~\cite{chen2018neural} different classes of neural networks are interpreted as different time discretization schemes of ordinary differential equations (ODEs). These studies lead to the formulation and the investigation of the so-called neural differential equations. There are also studies on application of mean-field and kinetic theory to ResNets, e.g.~see~\cite{araujo2019mean,mei2018mean,sirignano2019mean}. For instance, in~\cite{sirignano2019mean} the authors consider the mean-field formulation of a single hidden layer ResNet in the limit of infinitely many neurons.

We consider a particular structure of a ResNet where all the layers are characterized by the same number of neurons $N$, which is fixed by the dimension of the input measurements. This assumption underlies also the derivation of neural differential equations~\cite{chen2018neural}. Throughout the paper, we name ResNets with this structure as simplified residual neural networks (SimResNets). Networks with a \emph{similar} type of structure have been proved to satisfy the universal approximation theorem for different class of functions~\cite{UniversalApproximator,kidger2020universal} and also for probability distributions~\cite{LuLu2020}. Moreover, SimResNets have been successfully tested on engineering applications~\cite{GebhardtTrimborn,Bobzin2021}.

In this work, in contrast to~\cite{sirignano2019mean}, we do not consider the limit of infinitely many neurons. Instead, using as starting point the differential formulation of the SimResNet, we compute the mean-field limit in the number of input data which lead to a hyperbolic Vlasov-type PDE for the evolution of the distribution of the measurements. In this microscopic to kinetic limit process we do not solve any optimization problem to determine the parameters of the network.
Insights into the choice of the parameters are recovered by analyzing the forward propagation of the mean-field formulation of the SimResNet. In particular, we analyze steady states and study the properties of the moment model in the simple setting of a linear activation function and one-dimensional input data, deriving conditions on the parameters to solve certain aggregation or clustering phenomena.
Furthermore, we perform a sensitivity analysis in order to obtain a novel update algorithm for the parameters when a perturbation either in the input or in the target distribution is considered. This leads to an easier and computationally cheaper algorithm in comparison to standard back-propagation algorithms, and therefore can be of utmost advantage in real world applications. The robustness, in terms of perturbations in the input data, of the SimResNet has been recently studied in~\cite{GersterTrimbornVisconti} using spectral methods.

In addition, we propose a Boltzmann-type formulation of a reformulation of the SimResNet which includes noisy dynamics. This modification is motivated by stochastic neural networks with stochastic output layers which capture uncertainty over activations, see e.g.~\cite{Goldberger2017TrainingDN,NIPS2017_7096,Tran2019BayesianLA,DBLP:conf/icip/YouYLX019}. Long time behavior of such Boltzmann-type equations can be conveniently studied in the grazing limit regime which naturally leads to a Fokker-Planck equation where it is possible to obtain non trivial steady state distributions~\cite{PareschiToscani2006,PareschiToscaniBOOK,Toscani2006}.
Steady states obtained in the Fokker-Planck asymptotic depend on the parameters of the network which have to be chosen in order to fit the target distribution. In other words, we replace the concept of training with the task of fitting a distribution by imposing conditions on the parameters of the network and on the choice of the activation function.

The outline of the paper is as follows. In Section 2 we briefly review the microscopic ResNet model and we introduce the simplifications leading to the SimResNet. Then, the time continuous limit of the SimResNet is formally computing, leading to a neural differential equation. We derive the corresponding mean-field limit and we investigate steady states, moment properties and perform sensitivity analysis. In Section 3 the Boltzmann-type neural network model is presented, together with an asymptotic limit leading to a Fokker-Planck equation. We analyze the ability of the Fokker-Planck equation to describe given target distributions depending on the choice of parameters and activation functions. In Section 4 we conduct several artificial numerical tests which validate the previous analysis. Especially, we take into account the solution of two classical machine learning problems, classification and regression. We conclude the paper in Section 5 with a brief summary and an outlook on future research perspectives.

\section{Mean-Field Formulation of Residual Neural Networks with a Simplified Structure}

Let us consider input signals characterized by $d$ features. A feature is one type of measured data, e.g.~the temperature of a tool, the size of a vehicle, the color intensity of the pixels of an image. 
Without loss of generality we assume that the value of each feature is one-dimensional and thus the input signals are given by $\boldsymbol{x}_i^0\in\R^d$, $i=1,\dots,M$. Here, $M$ denotes the number of measurements or input signals.
We consider $L$ hidden layers, labeled by $\ell=1,\dots,L$, and in each layer the number of neurons is given by $N_\ell$, $\forall \ell=1,\dots,L$. We denote with $\ell=0$ and $\ell=L+1$ the input layer and the output layer, respectively. In particular, $N_0=d$ and $\boldsymbol{x}_i(\ell)\in\R^{N_\ell}$ is the propagation of the $i$-th input signal through the network at the $\ell$-th layer.
	
Given an input signal $\boldsymbol{x}_i^0\in\R^d$, a ResNet defines precise microscopic dynamics for the propagation of the activation energy of each neuron through the layers as~\cite{He2015DeepRL}
\begin{equation} \label{eq:defNN}
\begin{cases}
\boldsymbol{x}_i(\ell+1)=\boldsymbol{x}_i(\ell)+\Delta t\ \sigma\left( \boldsymbol{w}(\ell)\ \boldsymbol{x}_i(\ell)+\boldsymbol{b}(\ell) \right), \quad \ell=0,\dots,L\\
\boldsymbol{x}_i(0)=\boldsymbol{x}_i^0
\end{cases}
\end{equation}
for each $i=1,\dots,M$.
Here, $\boldsymbol{w}(\ell) \in\R^{N_{\ell+1}\times N_\ell}$ are the weights and $\boldsymbol{b}(\ell)\in\R^{N_{\ell+1}}$ the bias. These define the parameters of the network to be optimized. Instead, $\sigma: \R\to \R$ denotes the activation function which is applied component wise. Examples of activation functions are the identity function $\sigma_I(x) = x$, the so-called ReLU function $\sigma_R(x) = \max\{0,x\}$, the sigmoid function $\sigma_S(x) = \frac{1}{1+\exp(-x)}$, the hyperbolic tangent function $\sigma_T(x) = \tanh(x)$, and the growing cosine unit function $\sigma_{\text{GCU}}(x)=x\cos(x)$. For the time being, $\Delta t\in\R^+$ is a multiplicative factor.

A crucial part in applying a neural network is the training procedure. By training one aims to minimize the distance between the output $\boldsymbol{x}_i(L+1)$ of the neural network at the final layer and a given target for the signal $i$, here denoted with $\boldsymbol{h}_i$. Mathematically speaking one aims to solve the minimization problem 
$$
	\min_{\boldsymbol{W},\boldsymbol{B}} \sum_{i=1}^M D(\boldsymbol{x}_i(L+1),\boldsymbol{h}_i)
	%\left\| \boldsymbol{x}_i(L+1)- \boldsymbol{h}_i\right\|^2_2,
$$
where $D$ is a given loss function, and where $\boldsymbol{W}$ and $\boldsymbol{B}$ are the collection of the weights and bias, respectively. Several choices for the loss function are possible~\cite{LossFunctions}. %for the sake of the presentation the mean squared $L^2$ distance is considered as loss function,
The procedure can be computationally expensive on the given training set. Most famous examples of such an optimization procedure are back-propagation algorithms based on the stochastic gradient descent~\cite{haber2018look} or the ensemble Kalman filter~\cite{Kovachki2019,Watanabe1990,Alper}.
In the following we tackle the problem of the training of the network from a different perspective, which does not involve the solution of such an optimization procedure.

\subsection{Simplified Residual Neural Networks and Time Continuous Limit} \label{ssec:simresnet}

In~\eqref{eq:defNN} the layers define a discrete structure within the ResNet. We interpret the layers as discrete times where the propagation of the input signal through the dynamics of the network is evaluated. To this end the parameter $\Delta t$ is seen as the time step of the time discretization.

In order to formally compute the time continuous limit of~\eqref{eq:defNN} we need to introduce the following assumption. We assume that the number of neurons is identical in each layer and is defined by the size of the input signal, i.e.~we take $N_\ell=N=d$, $\forall\,\ell=1,\dots,L+1$. We will refer to ResNet with this structure as Simplified Residual Neural Network (SimResNet). We recall that a similar assumption underlies the derivation of neural differential equations, see e.g.~\cite{ruthotto2018deep,chen2018neural}. In this way, the ResNet~\eqref{eq:defNN} can be seen as an explicit Euler discretization of an underlying differential equation. Namely, in the limit $\Delta t\to 0^+$ and $L\to \infty$, \eqref{eq:defNN} formally leads to
\begin{equation} \label{eq:simplifiedNN}
\begin{cases}
\displaystyle{\frac{\mathrm{d}}{\mathrm{d}t}} \boldsymbol{x}_i(t) = \sigma\left( \boldsymbol{w}(t)\ \boldsymbol{x}_i(t) + \boldsymbol{b}(t) \right), \quad t>0\\
\boldsymbol{x}_i(0) = \boldsymbol{x}_i^0,
\end{cases}
\end{equation}
for each $i=1,\dots,M$, and where now $\boldsymbol{w}(t)\in\R^{d\times d}$, $\boldsymbol{b}(t)\in\R^d$. We will refer to~\eqref{eq:simplifiedNN} as the continuous SimResNet or the neural differential equation associated to the SimResNet. Existence and uniqueness of a solution to~\eqref{eq:simplifiedNN} is guaranteed as long as the activation function $\sigma$ satisfies the Lipschitz condition.

The simplification leading to a SimResNet, which allows to write~\eqref{eq:simplifiedNN}, is not only beneficial for the derivation of the time continuous limit and of the mean-field equation in the next section. In fact, networks with a similar structure have been already considered, e.g.~in~\cite{UniversalApproximator,LuLu2020,kidger2020universal}, showing that universal approximation theorems for different classes of functions and for distributions hold true. In addition, the performance of the SimResNet has been successfully investigated for engineering applications~\cite{GebhardtTrimborn,Bobzin2021}.

\subsection{Mean-Field Limit}

In this section we derive the PDE description of the forward propagation of the time continuous formulation of the SimResNet. To this end, we follow a Liouville-type approach by computing the corresponding mean-field limit of~\eqref{eq:simplifiedNN} in the number of measurements $M$.
In this regime the dynamic of~\eqref{eq:simplifiedNN} can be described by a hyperbolic Vlasov-type PDE:
\begin{equation} \label{eq:meanfield}
\partial_t g(t,\boldsymbol{x}) + \nabla_{\boldsymbol{x}} \cdot \Big( \sigma\big( \boldsymbol{w}(t)   \boldsymbol{x} + \boldsymbol{b}(t) \big) g(t,\boldsymbol{x}) \Big)=0.
%\sum\limits_{k=1}^d \partial_{x_k} \left( \sigma\left(  \frac{1}{d} \sum\limits_{j=1}^d   w_{k, j}(t)   x^j  + b_j(t)       \right)    g(t,\boldsymbol{x})  \right)=0,
\end{equation}
Equation~\eqref{eq:meanfield} describes the evolution of the compactly supported probability distribution function $g:\R^+_0\times\R^d\to\R^+_0$
with normalized initial condition
$$g(0,\boldsymbol{x})= g_0(\boldsymbol{x}), \quad \int\limits_{\R^d} g_0(\boldsymbol{x})\mathrm{d}\boldsymbol{x}=1.$$ 
The probability distribution $g_0(\boldsymbol{x})$ is assumed to be known since it corresponds to the distribution of the measured input data. Observe that~\eqref{eq:meanfield} preserves the mass, i.e.~$\int\limits_{\R^d} g(t,\boldsymbol{x}) \mathrm{d}\boldsymbol{x}=1$, $\forall\,t>0$.

The derivation of~\eqref{eq:meanfield} from~\eqref{eq:simplifiedNN} follows classical arguments and is shortly summarized hereinafter.
We define the empirical measure of the solution vector $\boldsymbol{X}=(\boldsymbol{x}_1,..., \boldsymbol{x}_M)^T\in \R^{dM}$ as
$$
	\mu^M_{\boldsymbol{X}}(t,\boldsymbol{x}) = \frac1M \sum\limits_{k=1}^M \delta(\boldsymbol{x}-\boldsymbol{x}_k(t))%\cdot ... \cdot \delta(x^d-x_k^d(t))
$$
and use it to connect the solution of the large system of differential equations~\eqref{eq:simplifiedNN} to the PDE. A straightforward computation shows that
$\mu^M_{\boldsymbol{X}}(t,\boldsymbol{x})$ is a weak solution of the equation~\eqref{eq:meanfield}. We refer e.g.~to~\cite{CarrilloFornasierToscaniVecil2010,Golse,jabin2014review,DegondMotsch2008} for further details. Using rigorous arguments, based on the Dobrushin's inequality, it is possible to show that the empirical measure converges to the solution of the mean-field equation~\eqref{eq:meanfield}. Here, we follow the presentation in~\cite{Golse} where the convergence is obtained in the space of probability measures $\mathcal{P}(\R^d)$ using the $1$-Wasserstein distance, which is defined as follows.
	\begin{definition}
		Let $\mu$ and $\nu$ two probability measures on $\R^d$. Then the 1-Wasserstein distance is defined by
		\begin{align}
		W(\mu,\nu):= \inf\limits_{\pi\in \mathcal{P}^*(\mu,\nu)} \int\limits_{\R^d} \int\limits_{\R^d} |\xi-\eta| \mathrm{d}\pi(\xi,\eta),
		\end{align}
		where $\mathcal{P}^*$ is the space of probability measures on $\R^d\times\R^d$ such that the marginals are $\mu$ and $\nu$ i.e.
		$$
		\int\limits_{\R^d} \mathrm{d}\pi(\cdot, \eta)= \mathrm{d}\mu(\cdot),\quad \int\limits_{\R^d} \mathrm{d}\pi(\xi, \cdot)= \mathrm{d}\nu(\cdot).
		$$
	\end{definition}
	\begin{theorem}
		Assume that the activation function $\sigma$ of the  microscopic system \eqref{eq:simplifiedNN} is Lipschitz continuous with Lipschitz constant $C_L>0$.
		Let $g_0(\boldsymbol{x})$, initial condition of~\eqref{eq:meanfield}, be a probability measure with finite first moment such that 
		$$
		W(\mu^M_0, g_0)\to 0,\ \text{as}\  M\to\infty.
		$$
		Then the Dobrushin's stability estimate 
		$$
			W(\mu^M(t),g(t)) \leq \exp(C_Lt) W(\mu^M_0, g_0)
		$$
			is satisfied and thus
		$$
			W(\mu^M(t),g(t))\to 0,\ \text{as}\  M\to\infty.
		$$
	\end{theorem}
%	\begin{proof}
%		We only sketch the main steps of the proof and refer e.g.~to \cite{Golse} for details. 
%		As first step one defines the characteristic equations of the mean field neural network model \eqref{eq:meanfield}. 
%		These characteristic equations are measure dependent and one usually uses the push-forward operator in order to be able to derive the Dobrushin stability estimate.
%		The existence of the solution of the characteristic equations ban be shown with the help of the Lipschitz constant and the corresponding fixed point operator. 
%		As next step one considers the distance of two measures and again uses the fixed point operator and the Lipschitz continuity  in order to bound the distance. Finally, one can apply the Gr\"{o}nwall 
%		inequality and obtains the Dobrushin stability estimate. 
%\end{proof}
For the detailed proof of the above result we refer to~\cite{Golse}. In view of this result, we stress the fact that the mean-field interpretation of the SimResNet is accurate when the number of measurements is very large. Moreover, equation~\eqref{eq:meanfield} provides only a statistical information on the neural network propagation, and therefore any information on the network output of a precised measurement is lost. In other words, in this limit we provide a statistical interpretation of the learning task.

In the following we discuss solution properties of~\eqref{eq:meanfield}. If the flux function $F(t,\boldsymbol{x},u):=  u \, \sigma\left(  \boldsymbol{w}(t)   \boldsymbol{x} + \boldsymbol{b}(t) \right) $ fulfills
\begin{equation}\label{RegNN}\begin{aligned}F\in C^2(\R_+\times\R^d\times \R; \R) \\ \partial_u F\in L^{\infty}(\R_+\times\R^d\times \R; \R) \\ \partial_u \mathrm{div}_x(F)\in L^{\infty}(\R_+\times\R^d\times \R; \R)
\end{aligned}\end{equation}
then the hyperbolic conservation law~\eqref{eq:meanfield} admits a unique weak entropy solution in the sense of Kr\^uzkov for any initial data $g_0\in L^{\infty}\cap L^1$, see~\cite{colombo2008stability}, and can be solved pointwise by the method of characteristics. In particular, the following result characterizes the steady state solution of~\eqref{eq:meanfield}.

\begin{proposition}\label{Prop::Steady}
Let $g:\R^+_0\times\R^d\to\R^+_0$ be the compactly supported weak solution of the mean-field equation~\eqref{eq:meanfield}. Assume that the activation function $\sigma:\R\to\R$ have $n$ zeros $z_i$, i.e.~$\sigma(z_i)=0$ for $i=1,\dots,n$. Let $\boldsymbol{b}^{\infty}= \lim_{t\to\infty} \boldsymbol{b}(t)$ and $\boldsymbol{w}^{\infty}= \lim_{t\to\infty} \boldsymbol{w}(t)$ exist and be finite. Moreover, assume that $\boldsymbol{w}^\infty$ has maximum rank. Then
\begin{equation} \label{eq:weakSS}
g^{\infty}(\boldsymbol{x}) = \sum_{i=1}^{n^d} \rho_i \delta(\boldsymbol{x}-\boldsymbol{y}_i)
\end{equation}
is a steady state solution of~\eqref{eq:meanfield} in the sense of distributions provided that $\boldsymbol{y}$ solves $\boldsymbol{w}^\infty \boldsymbol{y}  + \boldsymbol{b}^\infty=\boldsymbol{0}$, with $\rho_i\in[0,1]$, $\forall\,i=1,\dots,n$, and $\sum_{i=1}^{n^d} \rho_i=1$.
\end{proposition}
\begin{proof}
For a test function $\phi(\boldsymbol{x})\in C_0^{\infty}(\R^d)$ the weak steady state $g^\infty(\boldsymbol{x})$ of~\eqref{eq:meanfield} solves
\begin{align}
\int\limits_{\R^d}  \nabla_{\boldsymbol{x}}\phi(\boldsymbol{x})  \sigma\left( \boldsymbol{w}^{\infty}   \boldsymbol{x}  + \boldsymbol{b}^{\infty} \right) g^{\infty}(\boldsymbol{x}) =0. \label{SteadyMF}
%\sum\limits_{k=1}^d\int  \partial_{x_k}\phi(\boldsymbol{x})\  \sigma\left(  \frac{1}{d} \sum\limits_{j=1}^d   w_{k, j}^{\infty}   x^j  + b^{\infty}_j       \right)    g_{\infty}(\boldsymbol{x})\  dx^1...dw^d =0, \label{SteadyMF}
\end{align}
If $g^{\infty}(\boldsymbol{x})$ is defined by~\eqref{eq:weakSS}, then equation~\eqref{SteadyMF} is satisfied only if $\boldsymbol{y}_i$ is the solution to the system $\boldsymbol{w}^\infty \boldsymbol{y}  + \boldsymbol{b}^\infty=\boldsymbol{z}$, with $\boldsymbol{z}$ any disposition of $d$ elements with repetition of the $n$ zeros of the activation function $\sigma$.
\end{proof}

Proposition~\ref{Prop::Steady} has an immediate consequence on the choice of the activation function. Unless $\sigma$ has more than one zero, the steady state of the mean-field equation is characterized by a single Dirac delta. This occurs choosing, e.g., the identity activation function $\sigma_I(x)$ or the hyperbolic tangent activation function $\sigma_T(x)$. Instead, choosing the growing cosine unit activation function $\sigma_{\text{GCU}}(x)$ it is possible to recover a linear combination of Dirac delta at steady state.

\subsection{Moment Properties of the One-Dimensional Mean-Field Equation}\label{Moment}

Among the problems that machine learning aims to solve, classification refers to a predictive task where a class label is identified for a given input data. An easy example is provided by the classification of emails as ``spam'' or ``not spam''. Labels can be also multiple. Therefore, in classification problems one typically aims to cluster a data set in given classes. From a mathematical perspective, and in particular using the mean-field formulation of the neural network, classification can be seen as the problem of driving the distribution of the input data to a distribution of Dirac delta functions. Proposition~\ref{Prop::Steady} already states that there exist Dirac delta steady solutions of the mean-field equation.

Here, we aim to characterize choices of the parameters to obtain this class of steady states by analyzing the moments in a simple setting. We recall, in fact, that it is not the aim of this work to determine the parameters by means of a classical training procedure.

We focus on one dimensional input data, i.e.~$d=1$. Then the mean-field equation~\eqref{eq:meanfield} reduces to
\begin{equation} \label{eq:meanfield1D}
\partial_t g(t,x) + \partial_x \Big(  \sigma\big(w(t) x+ b(t)\big) g(t,x)  \Big) = 0.
\end{equation}
We define the $k$-th moment, $k\geq0$, and variance of the probability distribution $g$ by
\begin{equation} \label{eq:momentDef}
m_k(t):= \int\limits_{\R} x^k\ g(t,x)\  \d x,\quad \mathbb{V}(t)= m_2(t)-(m_1(t))^2.
\end{equation}
Clearly, the possibility to obtain a moment model is solely determined by the shape of the activation function $\sigma(\cdot)$.
We study properties of the solution $g$ with respect to characterizations of $w$ and $b$. The properties are specified below, based on the expected value and energy of the solution to the mean-field equation.

\begin{definition} \label{def:concentration}
We say that the solution $g(t,y)$, $(t,y)\in\R^+\times\R$, to the mean-field equation~\eqref{eq:meanfield1D} is characterized by
\begin{itemize}
	\item[(i)] local energy bound if $$
	m_2(0) > m_2(t),
	$$
	holds at a fixed time $t$;
\item[(ii)] energy decay if
	$$
	m_2(t_1) >  m_2(t_2),
	$$
	holds for any $t_1<t_2$, i.e.~if the energy is decreasing with respect to time;
\item[(iii)] local aggregation if
	$$
	\mathbb{V}(0) > \mathbb{V}(t),
	$$
	holds at a fixed time $t$;
\item[(iv)] aggregation if
	$$
		\mathbb{V}(t_1)>  \mathbb{V}(t_2),
	$$
	holds for any $t_1<t_2$, i.e.~if the variance is decreasing with respect to time;
\item[(v)] concentration or clustering if
	$$
		\lim_{t\to\infty} \mathbb{V}(t) = 0,
	$$
	i.e.~if the variance vanishes in the long time behavior.
	\end{itemize}
\end{definition}

We observe that if the first moment is conserved in time, then definition of local energy bound, i.e.~(i), is equivalent to local aggregation, i.e.~(iii), and definition of energy decay, i.e.~(ii), is equivalent to aggregation, i.e~(iv). In terms of a residual neural network, satisfying concentration or aggregation phenomena means that we can recover output distributions with decreasing variance with respect to the input distribution. In particular, clustering is important for classification tasks, as specified before.

\subsubsection{The Case of the Identity Activation Function}

A simple computation reveals that the 0-th moment is conserved, i.e.~$m_0(t)=1$ holds for all times $t\geq0$, as we expect since~\eqref{eq:meanfield} is a conservation law. Instead, for $k\geq 1$ the behavior in time of the corresponding moment is prescribed by the following linear ordinary differential equation:
%\todo{In multidimension we have the same system but the solution depends on whether $w(t)$ commutes with $\Phi_k$.}
\begin{align}\label{eq::momentODE}
\frac{\d}{\d t} m_k(t) = k\ \Big( w(t)\ m_k(t)+ b(t)\ m_{k-1}(t) \Big),\quad  m_k(0)=m_k^0.
\end{align}
Notice that the $k$-th moment only depends on the $(k-1)$-th moment. It is then possible to solve the moment equations iteratively with the help of the separation of variables formula obtaining
\begin{equation} \label{eq:momentSol}
m_k(t)= e^{\Phi_k(t)} \left( m_k(0) +  k \int\limits_0^t   e^{-\Phi_k(s)}  b(s) m_{k-1}(s) \d s \right),
\end{equation}
where
$$
	\Phi_k(t):=k \int\limits_0^t w(s) \d s.
$$

We study conditions on the parameters to observe the phenomena described in Definition~\ref{def:concentration}.

\begin{proposition} \label{prop1}
Assume that the bias is identical to zero, namely $b(t)\equiv 0$, $\forall\,t\geq 0$. Then we obtain
\begin{enumerate}
\item[(a)] local energy bound if $\Phi_1(t)<0$ at a fixed time $t$;
\item[(b)] energy decay if and only if $w(t)<0$ for all $t>0$;
\item[(c)] clustering if and only if  $\lim\limits_{t\to\infty} \Phi_1(t)=-\infty$.
In particular the steady state is distributed as a Dirac delta centered at $x=0$. 
\end{enumerate}
\end{proposition}
\begin{proof}
If $b(t)\equiv 0$ then~\eqref{eq:momentSol} simplifies to 
\begin{equation} \label{eq:momentZeroBias}
	m_k(t)= m_k(0) e^{\Phi_k(t)}
\end{equation}
and thus the first and the second moment are identical except the given initial conditions. Then we can easily apply the definitions of energy bound, energy decay and clustering to prove the statement. 
\end{proof}

The previous result suggests that $b(t)\equiv 0$ is a condition which allows to solve clustering problems at the origin, independently on the initial first moment. The following corollary, instead, establishes equivalence of the phenomena defined in Definition~\ref{def:concentration} under the hypothesis of Proposition~\ref{prop1}.

\begin{corollary}
Assume that the bias is identical zero, namely $b(t)\equiv 0$, $\forall\,t\geq 0$. Then 
\begin{enumerate}
\item[(a)] if local energy bound exists at a time $t$ we have local aggregation;
\item[(b)] if energy decay holds we have aggregation.
\end{enumerate}
\end{corollary}
\begin{proof}
	We start proving the first statement. First we observe that~\eqref{eq:momentZeroBias} is still true, since by hypothesis we assume that the bias is zero. Due to the definition of local aggregation we need to verify that $\mathbb{V}(0) > \mathbb{V}(t)$ for a fixed time $t$. Using the definition of the variance, local aggregation is implied by
	$$
		m_2(0) (1-e^{\Phi_2(t)}) > m_1(0)^2 (1-e^{2\Phi_1(t)}).
	$$
	For the second part, we observe that aggregation phenomena is verified if $\mathbb{V}(t_1) > \mathbb{V}(t_2)$ for any $t_1 < t_2$. This is equivalent to 
	\begin{align*}
	&m_2(t_1)-m_1(t_1)^2 > m_2(t_2)- m_1(t_2)^2\\
	\iff &  m_2(t_1)-m_2(t_2) > (m_1(t_1)-m_1(t_2))\ (m_1(t_1)+m_1(t_2))\\
	\iff & m_2(t_1) > m_1(t_1)^2.
	\end{align*}

\end{proof}

Conservation of the first moment is guaranteed by choosing $b(t):= -w(t) m_1(t)$. See the following results.

\begin{proposition}\label{prop2}
Let the bias be $b(t):= -w(t) m_1(t)$, $\forall t\geq0$. Then the first moment $m_1$ is conserved in time and we obtain 
\begin{enumerate}
 \item[(a)] local energy bound if $\Phi_2(t)<0$ at a fixed time $t$;
 \item[(b)] clustering phenomenon if $w(t)<0$ holds for all $t\geq 0$. In particular the steady state is distributed as a Dirac delta centered at $x=m_1(0)$.
\end{enumerate}
\end{proposition}
\begin{proof}
	The solution formula for the second moment is
	$$
		m_2(t)= e^{ \Phi_2(t) } (m_2(0)-m_1(0))^2) + m_1(0)^2 = e^{ \Phi_2(t) } \mathbb{V}(0) + m_1(0)^2.
	$$
	Then we have local energy bound if
	$$
		\mathbb{V}(0) \left( e^{ \Phi_2(t) } -1 \right) < 0
	$$
	which is satisfied assuming that $\Phi_2(t)<0$ at a fixed time $t$. For the second statement we observe that concentration to a delta is also implied by $\Phi_2(t)<0$ for all times $t$, so that $m_2(t) \to m_1(t)^2$ for $t\to\infty$.
%	$\frac{\d}{\d t} \mathbb{V}(g(t)) < 0$, for all times. Or, equivalently, by $\frac{\d}{\d t} m_2(t) < 0$, for all times. We have
%	$$
%		\frac{\d}{\d t} m_2(t) = w(t) \mathbb{V}(t) < 0
%	$$
	This occurs if $w(t)<0$ for all times.
\end{proof}

We aim to discuss the impact of the variance on aggregation and concentration phenomena. 
This is especially interesting if we focus on a finite time horizon and we aim to know if $\mathbb{V}(T)\leq V$ for some tolerance $V>0$ and time $T>0$. In applications this level would be determined by the variance of the target distribution and allows to avoid over-fitting phenomena.

\begin{corollary}\label{cor2}
If the bias is identical to zero, namely $b(t)\equiv 0$, $\forall\,t\geq0$, then the energy at time $t>0$ is below tolerance $V>0$ if
$$
	\Phi_2(t) < \ln\left(\frac{V}{m_2(0)}\right)
$$
is satisfied. Instead, the variance is below the level $V>0$ if
$$
	\Phi_2(t) < \ln\left( \frac{V}{\mathbb{V}(0)} \right)
$$
holds.

Similarly, if the bias fulfills $b(t):= -w(t) m_1(t)$, then the energy at time $t>0$ is below the level $V>0$ if
$$\Phi_2(t)< \ln\left(\frac{V-m_1(0)^2}{ \mathbb{V}(0) } \right)$$
 is satisfied  provided that $V> m_1(0)^2$ holds. Instead, the variance at time $t>0$ is below the level $V>0$ if
$$\Phi_2(t)< \ln\left(\frac{V}{ \mathbb{V}(0) } \right)$$
 is satisfied  provided that $V>0$ holds.
\end{corollary}

\begin{remark}
In general it is not possible to obtain a closed moment model in the case of the sigmoid $\sigma_S(x)$ or hyperbolic tangent $\sigma_T(x)$ activation function.
Nevertheless one might approximate both activation functions by the linear part of their series expansion:
\begin{align*}
\sigma_S(x)\approx \frac{1}{2}+ \frac{x}{4}, \quad \sigma_T(x)\approx  x,
\end{align*}
and apply similar analysis discussed above for the identity activation function.
\end{remark}

\begin{remark}
In the case of the ReLU activation function we decompose the moments $k\geq 1$ in two parts
$$
m_k(t) = \int\limits_{\Omega^+(t)} x^k g(t,x) \d x + \int\limits_{\Omega^-(t)} x^k g(t,x) \d x,
$$
and we define
$$
m_k^+(t):= \int\limits_{\Omega^+(t)} x^k g(t,x) \d x, \quad 
m_k^-(t):= \int\limits_{\Omega^-(t)} x^k g(t,x) \d x
$$
where $\Omega^+(t) := \{ x\in\R \ | \ x > -\frac{b(t)}{w(t)}  \}$ and $\Omega^-(t) := \R\setminus \Omega^+(t)$.

Let us define $a(t) = -\frac{b(t)}{w(t)}$, then using the Leibniz integration rule we compute
\begin{equation} \label{eq:mMinusEvolution}
	\frac{\d}{\d t} m_k^-(t) = a(t)^k g(t,a(t)) \frac{\d}{\d t} a(t)
\end{equation}
and

\begin{align}\label{eq:mPlusEvolution}
\frac{\d}{\d t} m_k^+(t) =& -a(t)^k g(t,a(t)) \frac{\d}{\d t} a(t) + k w(t) m_k^+(t) + k b(t) m_{k-1}^+(t)\\
	& + a(t)^{k+1}\ w(t)\ g(t,a(t)) + a(t)^k\ b(t)\ g(t,a(t)) .
\end{align} 

Consequently, the evolution equation for the $k$-th moment cannot be expressed by a closed formula since it depends on the partial moment on $\Omega^+(t)$ and boundary conditions:
\begin{equation} \label{eq:mEvolution}
	\frac{\d}{\d t} m_k(t) = k \left( w(t) m_k^+(t) + b(t) m_{k-1}^+(t) \right)+ a(t)^{k+1}\ w(t)\ g(t,a(t)) + a(t)^k\ b(t)\ g(t,a(t)).
\end{equation}

In the case of constant weights and bias the equality $\dot{m}_k= \dot{m}_k^+$ holds. Notice also that the above discussion simplifies in the case of vanishing bias, and it becomes equivalent to the case when the activation function is the identity function. In fact, if $b(t) \equiv 0$, $\forall\, t\geq0$, then the set $\Omega^+$ switches to be $(-\infty, 0)$ or $(0,\infty)$, depending on the sign of the weight $w(t)$. Moreover, thanks to~\eqref{eq:mMinusEvolution} we immediately obtain that $\frac{\d}{\d t} m_k^-\equiv 0$ holds true and thus $\frac{\d}{\d t} m_k(t)= \frac{\d}{\d t} m_k^+(t)$ is satisfied for all $t \geq 0$. Hence, the evolution equation~\eqref{eq:mEvolution} reduces to the case \eqref{eq::momentODE} and same computations can be performed.
\end{remark}

\subsection{Forward Re-training Algorithm of the Mean-Field Neural Network}\label{sensitivity}

We introduce a novel re-training algorithm for the weights and bias of the mean-field formulation of the SimResNet. Classical back-propagation algorithms, typically used for the training process of a network, suffer from the high computational cost. We derive a forward algorithm with the help of the sensitivity analysis. More precisely, we define a loss function, coupled to the mean-field equation, which we aim to minimize. In this setting the training of a the mean-field SimResNet boils down to an optimal control problem. Therefore, we employ adjoint calculus in order to compute the sensitivity quantities of the loss function, i.e.~the gradients of the loss function with respect to the weights and the bias. This sensitivity quantities lead to a gradient-type algorithm which computes the update of the parameters in the case of a perturbed target or new initial condition. As stated before, the benefit of this algorithm is the possibility to avoid a retraining via back-propagation in the case of perturbations.

The quantity of interest is the distance of the function $g$ at a finite time $T$ to the target distribution $h$. We consider the following loss function:
\begin{align}
	D(T;w,b,g_0) := \frac{1}{2} \int\limits_\R |g(T,x)-h(x))|^2 \d x. \label{lossFunc}
\end{align}
The minimization of~\eqref{lossFunc} subject to the mean-field equation defines an optimal control problem:
\begin{align*}
\arg\min\limits_{w,b} \ \ &D(T;w,b,g_0)\\
\text{s.t.} \ \ &\partial_t g(t,x) + \partial_x \Big(  \sigma\big(w(t) x+ b(t) \big) g(t,x)  \Big) = 0,\\
&g(0,x)=g_0(x).
\end{align*}
We may expect that training was expensive and will not necessarily be done again if input or target changes. Therefore, it is of interest if the trained network, so $w$ and $b$, can be reused if $g_0$ and $h$ change. We compute the corresponding sensitivity quantities of the loss function $D$ with respect to the weights and the bias. This in turn can be used to apply a gradient step on $w$ and $b$. In other words, we use adjoint calculus to adjust the parameters to the perturbations introduced in $g_0$ and $h$.

The formal Lagrangian reads
\begin{align*}
 L(g,\lambda,w,b) =& D(T)-\int\limits_0^T \int\limits_{\R} g(t,x) \Big(\partial_t \lambda(t,x) + \sigma(w x +b) \partial_x \lambda(t,x) \Big) \d x \d t \\ &+ \int\limits_{\R} \left(\lambda(T,x) g(T,x)- \lambda(0,x) g(0,x)\right) \d x
\end{align*}
with Lagrange multiplier $\lambda(t,x)$. The corresponding adjoint equation can be obtained from the functional derivative of the Lagrangian:
\begin{align*}
 \partial_g L(g,\lambda,w,b) [\delta g] =& \int\limits_{\R} |g(T,x)-h(x)| \delta g(T,x) \d x - \int\limits_0^T \int\limits_{\R} \delta g \Big(\partial_t \lambda(t,x)+ \sigma(w x+b)  \partial_x \lambda(t,x) \Big) \d t \d x\\
  &+ \int\limits_{\R} \lambda(T,x) \delta g(T,x) \d x.
\end{align*}
Furthermore, we obtain the sensitivity quantities as 
\begin{align*}
\partial_w L(g,\lambda, w,b) [\delta w] &= \int\limits_0^T \delta w \int\limits_{\R} g(t,x) x \sigma^{\prime}(wx + b) \partial_x \lambda \, \d t \d x\\
\partial_b L(g,\lambda,w,b) [\delta b] &= \int\limits_0^T \delta b \int\limits_{\R} g(t,x) \sigma^{\prime}(wx + b) \partial_x \lambda \, \d t \d x. 
\end{align*}
Hence, the strong form of the formal first order optimality conditions are given by:
\begin{align*}
\partial_t \lambda(t,x) + \sigma(wx+ b) \partial_x\lambda(t,x) &= 0\\
\lambda(T,x) &= |g(T,x)-h(x)|\\
\partial_t g(t,x)+ \partial_x\Big(\sigma\big(w x+ b\big) g(t,x)\Big) &= 0\\
g(0,x) &= g_0(x)\\
\int\limits_{\R} g(t,x) x \sigma^{\prime}(wx + b) \partial_x\lambda \, \d x &= 0, \ \forall t\in[0,T]\\
\int\limits_{\R} g(t,x) \sigma^{\prime}(wx + b) \partial_x\lambda \, \d x &= 0, \ \forall t\in[0,T]
\end{align*}
where $\sigma'(x)$ is the derivative of the activation function. It is assumed that $\sigma$ is differentiable, i.e.~either $\sigma=\sigma_T$ or $\sigma=\sigma_S$.
Adjustment of optimal weights and bias can be then obtained via gradient step. The update of the parameters after a perturbation of the initial condition $g_0$ and the target $h$ is summarized in Algorithm~\ref{alg:sensitivity}.

\begin{algorithm}[t!]
	\begin{algorithmic}[1]
		\STATE Initially select the optimized parameters $w^0$ and $b^0$;
		\STATE Update $D$ by the new perturbed initial data $g_0$ and target $h$;
		\STATE Choose a tolerance $tol>0$;
		\WHILE{$|w^{k+1} - w^k|+ |b^{k+1} - b^k|<tol$}
		\STATE Compute the new updates as
		\begin{align*}
		& w^{k+1} = w^k -  \gamma \int\limits_{\R} g\ x\ \sigma^{\prime}\ \partial_x\lambda \, \d x  \\
		& b^{k+1} = b^k - \gamma\ \int\limits_{\R} g\ \sigma^{\prime}\ \partial_x\lambda \, \d x.
		\end{align*}
		\ENDWHILE
	\end{algorithmic}
	\caption{Forward update of the parameters of the mean-field SimResNet.}
	\label{alg:sensitivity}
\end{algorithm}

\section{Boltzmann-type Formulation of the SimResNet}

In general, a mean-field equation can be obtained as suitable asymptotic limit of a Boltzmann-type space homogeneous kinetic equation. We show that this is true also for the mean-field limit~\eqref{eq:meanfield} of the neural differential equation~\eqref{eq:simplifiedNN} by suitably defining instantaneous microscopic interactions emerging from the continuous dynamics. We point-out that in this approach the measurements are seen as particles.

In the case of one-dimensional measurements the mean-field equation can be derived from a linear Boltzmann-type equation. In fact, the system of ODEs~\eqref{eq:simplifiedNN} can be recast as the following interaction rule:
\begin{align}
x^*=x+\sigma\big(w(t) x +b(t)\big), \label{interactionBoltzmann}
\end{align}
where, by using the kinetic terminology, $x^*$ and $x$ are the post- and the pre-collision states, respectively. The post-interaction rule~\eqref{interactionBoltzmann} can be obtained from~\eqref{eq:simplifiedNN} by an explicit Euler discretization with time step $\Delta t=1$. Notice that although~\eqref{interactionBoltzmann} seems to not depend on the other particles, i.e.~the measurements, the interactions are hidden in the parameters $w$ and $b$ which are indeed chosen based on the whole set of data.

The weak form of the Boltzmann-type equation corresponding to~\eqref{interactionBoltzmann} reads
\begin{align}\label{eq::Boltzmann1D}
 \frac{\d}{\d t} \int\limits_\R \phi(x) g(t,x) \d x = \frac{1}{\tau} \int\limits_\R \Big( \phi(x^*)-\phi(x) \Big) g(t,x) \d x
\end{align}
where $\tau$ represents the interaction rate and $\phi\in C^\infty_0(\R)$ is a test function. Derivation and formulation of~\eqref{eq::Boltzmann1D} from dynamics like~\eqref{interactionBoltzmann} can be found in classical references on kinetic theory, e.g.~\cite{Arlotti2002567,BellomoMarsanTosin2013,PareschiToscaniBOOK}. We observe that since the post-interaction rule~\eqref{interactionBoltzmann} depends on other data only through $w$ and $b$, we obtain a linear Boltzmann-type collision kernel. In the present homogeneous setting, $\tau$ influences only the relaxation speed to equilibrium and thus, without loss of generality, in the following we take $\tau=1$.
% In the two-dimensional case we can interpret the microscopic ODE model as binary interactions:
%   \begin{align*}
%   x'= x + \sigma\left( \frac{w(t)\ x  +\bar{v}(t) y}{2}) + b(t) \right),\\
%  y' =y +  \sigma\left( \frac{v(t) \ y+\bar{w}(t)\ x}{2} + d(t) \right).
%   \end{align*}
% Then the weak form of the nonlinear Boltzmann-type equation reads:
% $$
% \frac{\d}{\d t} \int\limits_{\R^2} \phi(x,y)\ f(t,x,y)\ \d x \d y = \int\limits_{\R^2} \int\limits_{\R^2} [\phi(x',y')-\phi(x,y)]\ f(t,x,y)\  f(t,z,u)\ \d x \d y \d z \d u
% $$
% We aim to connect the Boltzmann type description to our mean field neural network model. This can be conveniently performed by standard Fokker-Planck type asymptotics 
% for example performed e.g.~in~\cite{PareschiToscani2006,PareschiToscaniBOOK,Toscani2006}. In fact,
The one-dimensional Boltzmann-type equation~\eqref{eq::Boltzmann1D} leads to the one-dimensional mean-field equation~\eqref{eq:meanfield1D} by suitable scaling, e.g.~see~\cite{PareschiToscani2006,PareschiToscaniBOOK,Toscani2006}.

\begin{remark}
	In the case of two-dimensional input signal, the SimResNet is structured with two neurons in each, cf.~Section~\ref{ssec:simresnet}. Then, we can reinterpret the differential formulation of the SimResNet~\eqref{eq:simplifiedNN} in the following binary interactions:
	\begin{align*}
		x^*= x + \sigma\left( \frac{w(t) x  +\bar{v}(t) y}{2} + b(t) \right),\\
		y^* =y +  \sigma\left( \frac{v(t) y+\bar{w}(t) x}{2} + d(t) \right).
	\end{align*}
	The interaction between the particles is now explicit and we have true binary interactions. This leads to a nonlinear space homogeneous Boltzmann-type equation.
%	Then the weak form of the nonlinear Boltzmann type equation reads:
%	$$
%		\frac{\d}{\d t} \int\limits_{\R^2} \phi(x,y)\ f(t,x,y)\ \d x \d y = \int\limits_{\R^2} \int\limits_{\R^2} [\phi(x^*,y^*)-\phi(x,y)]\ f(t,x,y)\  f(t,z,u)\ \d x \d y \d z \d u.
%	$$
\end{remark}

An advantage of the Boltzmann-type description~\eqref{eq::Boltzmann1D} is the possibility to study different asymptotic scales, in addition to the mean-field one, which can lead to equations with non-trivial steady states and that allows for an analytically characterization. These steady states are obviously parameterized by $w$ and $b$. Thus, it is possible to introduce a different concept of training in which the choice of the parameters of the network and of the activation function may be obtained by fitting the target distribution. 

In order to obtain non-trivial steady states of model \eqref{eq::Boltzmann1D} one may consider self-similar solutions~\cite{PareschiToscaniBOOK}:
$$
\bar{g}(t,x)=m_1(t) g(t, m_1(t) x). 
$$
In the the case of the identity activation function the first moment can be computed explicitly, as showed in Section~\ref{Moment}, and a Fokker-Planck type asymptotic equation can be derived. However, in order to study steady state profiles for arbitrary activation functions, we choose the following approach: we add noise to the microscopic interaction rule \eqref{interactionBoltzmann} and apply a grazing collision limit leading to a Fokker-Planck type equation.

%Before we start to discuss the Fokker-Planck asymptotics let us point out the limitations of this idea. Clearly does the performance of the neural network depend on the initial data as well. As pointed out previously, is the goal of a NN to minimize the distance between the density $f(t,x)$ to the target $h(x)$ in finite time. For sure, that can be connected to the steady state in the following way:
%\begin{align*}
%|  f(t,x)- h(x)   | = &| f(t,x)   -  f_{\infty}+  f_{\infty}  -  h(x)  | \leq    |f(t,x) - f_{\infty}(x) | + |  f_{\infty}  -  h(x)    |   \\
% &\leq   e^{-D t} |f(0,x)- f_{\infty}| + |  f_{\infty}  -  h(x)    |
%\end{align*}
%The exponential function is nothing else than the decay rate of our Fokker-Planck model to the steady state. We see that the error can be bounded  by the error between initial data and steady state solution of the Fokker-Planck model and by the error between target and steady state.  Obviously, we are only able to reduce the error between steady state and target distribution by analyzing the asymptotic behavior. Nevertheless we expect that this analysis provides new insights into
%NN and may improve the performance of NN in special situations. \\
%There are two possible approaches, starting from our Boltzmann type equations.

\subsection{Fokker-Planck Description of the Simplified Residual Neural Network}
% First we may add multiplicative noise to the one-dimensional microscopic interaction rule~\eqref{interactionBoltzmann} which means that the dynamic interacts with an uncertain background.
Let $\epsilon$ be a small parameter, weighting for the strength of the interactions. Let us modify the interaction~\eqref{interactionBoltzmann} as
\begin{equation} \label{eq:microRule}
	x^{*} = x + \epsilon \sigma(w(t) x + b(t) ) + \sqrt{\epsilon} K(x) \eta,
\end{equation}
where $\eta$ is a random variable with mean zero and variance $\nu^2$, and $K(x)$ is a diffusion function. For $K(x)=0$ \eqref{eq:microRule} corresponds to~\eqref{eq:defNN} with $\epsilon=\Delta t$ and to~\eqref{interactionBoltzmann} with $\epsilon=1$.
% The corresponding Boltzmann equation is
% $$
% 	\frac{\d}{\d t} \int\limits_\R \phi(x)\ f(t,x)\ \d x = \left\langle \int\limits_\R \left( \phi(x')-\phi(x) \right)\ f(t,x)\ \d x \right\rangle
% $$
% where $\langle\cdot\rangle$ is the expectation with respect the distribution of the noise $\eta$.

The introduction of noisy interactions as in~\eqref{eq:microRule} is motivated by the need of recovering a broader class of steady states, in addition to the trivial ones described by the mean-field equation. From the view-point of neural networks, this modeling assumption is inspired by the stochastic neural networks with stochastic output layers~\cite{Goldberger2017TrainingDN,NIPS2017_7096,Tran2019BayesianLA,DBLP:conf/icip/YouYLX019}. Stochastic neural networks consider random variations into the network, and one way to model these variations is using stochastic output layers which capture uncertainty over activations. It has been observed that this approach is useful for the optimization procedure, since the random fluctuations allow to escape from local minima. In the same spirit, we modify the deterministic interactions~\eqref{interactionBoltzmann} by describing layers with uncertainty about some state in its computation also in the forward propagation.

We point out that the~\eqref{interactionBoltzmann} is not linked to neural network structures presented in the literature. Rather, \eqref{interactionBoltzmann} is inspired by the existence of stochastic neural networks with stochastic output layers, which are defined by adding stochastic noise to the output of the neurons in each layer. The structure of the neural network corresponding to~\eqref{interactionBoltzmann} can be formulated as slight modification of~\eqref{eq:simplifiedNN} in terms of a system of stochastic differential equations with a noisy perturbation term weighted by $\frac{1}{\sqrt{\epsilon}}$. Discretization of this by Euler-Maruyama method with $\Delta t=\epsilon$ leads to~\eqref{eq:microRule}.

In the classical grazing collision limit $t = \epsilon t$, $\epsilon \to 0$, we recover the following Fokker-Planck equation for the scaled probability distribution $g$: 
\begin{equation} \label{eq:fokkerplanck}
	\partial_t g(t,x) + \partial_x \left[ \mathcal{B} g(t,x) - \mathcal{D} \partial_x g(t,x) \right] = 0
\end{equation}
where we define the interaction operator $\mathcal{B}$ and the diffusive operator $\mathcal{D}$ as
$$
	\mathcal{B} = \sigma(w(t)x+b(t)) - \frac{\nu^2}{2} \partial_x K^2(x), \quad \mathcal{D} = \frac{\nu^2}{2} K^2(x).
$$
% It is possible to show that the usual remaining term arising from second order Taylor expansions vanishes without further assumptions.
% \begin{remark}
% 	Mass conservation holds true under the assumption of a compactly supported distribution $f$ with fast decay to infinity.
% \end{remark}
The grazing collision limit is obtained by classical arguments and computations. Namely, starting from the linear Boltzmann-type equation~\eqref{eq::Boltzmann1D} one introduces a second order Taylor expansion of $\phi(x^*)$ around $x$. Studying the limit $\epsilon\to 0$ one gets~\eqref{eq:fokkerplanck}.

The advantage in computing the grazing collision limit is that the classical integral formulation of the Boltzmann collision term is replaced by differential operators. This allows a simple analytical characterization of the steady state solution $g^\infty(x)$ of~\eqref{eq:fokkerplanck}.
Provided the target can be well fitted by a steady state distribution of the Fokker-Planck equation, the weight, the bias and the activation function are immediately determined. This is a huge computational advantage in comparison to the classical training of neural networks.

In the following, we present examples of distributions that can be fitted by the solution of~\eqref{eq:fokkerplanck} in the large time behavior, finding conditions on the parameters and on the activation function.

\subsubsection{Steady State Characterization} \label{FPsteadystate}
Steady state solution of the Fokker-Planck equation~\eqref{eq:fokkerplanck} can be easily found as
% by solving the ODE
% $$
% 	f^\prime(x) = \frac{\mathcal{B}}{\mathcal{D}} f(x)
% $$
% by separation of variables. We obtain
\begin{equation} \label{eq:solutionFP}
	g^\infty(x) = \frac{C}{K^2(x)} \exp\left( \int \frac{2\sigma(w^\infty x + b^\infty)}{\nu^2 K^2(x)} \d x \right)
\end{equation}
where the constant $C\in\R$ is determined by mass conservation, i.e.~$\int\limits_\R g^\infty(x) \d x = 1$.
The existence and the explicit shape of the steady state is determined by the specific choice of the activation function $\sigma(\cdot)$, of the diffusion function $K(\cdot)$ and of the parameters $w^{\infty}$, $b^{\infty}$. 
% In the following we discuss whether $f_\infty(x)$ may fit classical distribution, under some modeling assumptions for the activation function $\sigma(\cdot)$ and the diffusion function $D(x)$.
%\begin{itemize}
%\item
We discuss below some typical cases.

If the target $h(x)$ is distributed as a Gaussian, then choosing $\sigma(x)=\sigma_I(x)$ and $K(x)=1$ the steady state~\eqref{eq:solutionFP} yields a suitable approximation of $h(x)$. In fact, we obtain 
$$
	g^\infty(x) = C \exp\left( \frac{w^\infty}{\nu^2} x^2 + 2 \frac{b^\infty}{\nu^2} x \right),
$$
which is a Gaussian provided that $w^\infty<0$ and $b^\infty = 0$. Moreover,
condition on mass conservation leads to
$$
	C = \frac{\sqrt{-\frac{w^\infty}{\nu^2}} \exp\left( \frac{(b^\infty)^2}{w^\infty \nu^2} \right)}{\sqrt{\pi}},
$$
which is defined for $w^\infty < 0$.% to guarantee converge of $\int\limits_\R g_\infty(x) \d x$.

If the target $h(x)$ is distributed as an inverse Gamma, then choosing $\sigma(x)=\sigma_I(x)$ and $K(x)=x$ the steady state~\eqref{eq:solutionFP} is given by
$$
  g^{\infty}(x) = \begin{cases}
  0, &x\leq 0,\\
  \frac{C}{x^{1+\mu}} \exp\left( \frac{\mu-1}{x}\ \frac{b^{\infty}}{w^{\infty}}  \right), &x>0,
  \end{cases}
$$
with $\mu:=1+\frac{2\ w^{\infty}}{\nu^2}$ and normalization constant
$$
  C= \frac{\Big((1-\mu) \frac{w^{\infty}}{b^{\infty}}\Big)^{\mu}}{\Gamma(\mu)},
$$
where $\Gamma(\cdot)$ denotes the Gamma function. Notice that we have to assume that $w^{\infty}<0$ and $b^{\infty}>0$ hold in order to obtain a distribution.
  
Let $\sigma(x)=\sigma_R(x)$ and $K(x)=x$, and, without loss of generality, assume $w^{\infty}<0$ so that $\sigma_R$ is identical zero on the set $\Omega:=\{ x\in\R| x\geq -\frac{b^{\infty}}{w^{\infty}} \}$. The steady state on $\Omega$ is given by
$g^{\infty}|_{\Omega}= \frac{c}{x^2}, c>0$ and can be extended on $\mathbb{R}$ by the Pareto distribution:
$$
  g^{\infty}= \begin{cases}
  -\frac{b^{\infty}}{w^{\infty} }\frac{1}{x^2}, &x\geq  -\frac{b^{\infty}}{w^{\infty}} ,\\
  0, & x<   -\frac{b^{\infty}}{w^{\infty}} .
  \end{cases}
$$
%   For the set $\R\setminus \Omega$ the ReLu activation function is the identity and it is possible to obtain an inverse gamma distribution as discussed previously. 

As last example, we notice that if the activation and the diffusion function are chosen as
$$
\sigma_N(x) := \left[\frac{1}{\delta} \left(\frac{x}{c}\right)^{\delta}-1\right] x,\ 0<\delta<1,\ c>0,\quad K(x)=x,
$$
respectively, then
it is possible to obtain a generalized Gamma distribution as steady state. 
This specific model has been discussed in~\cite{dimarco2019kinetic} and the exponential convergence of the solution to the steady state has been proven in~\cite{otto2000generalization}. 
This may motivate to choose the novel activation function $\sigma_N(\cdot)$ provided that the target data are distributed as a generalized Gamma distribution.

\section{Numerical Experiments}

In this section we present two classical applications for machine learning algorithms, namely classification and regression problems.
Using these applications we validate the theoretical observations of the moment model analysis in Section~\ref{Moment}. In addition, we test the update algorithm for the weights and the bias derived by the sensitivity analysis in Section~\ref{sensitivity}.
Finally, we show that the Fokker-Planck interpretation of the neural network is able to fit non trivial continuous probability distributions. We point-out that the weights and the bias we use for these experiments are not trained by usual training procedures and are constant in time. In fact, it is not the aim of this paper discussing training of the parameters for the PDE formulations of the neural network.

In all the numerical simulations we solve the PDE models presented in this work by using a third order finite volume scheme~\cite{CPSV:cweno}, which is briefly reviewed below. All the cases we consider for the numerical experiments can be recast in the following compact formulation:
\begin{equation} \label{eq:generalPDE}
    \partial_t u(t,x) + \partial_x F(u(t,x),t,x) = \frac{\nu^2}{2} \partial_{xx} u(t,x) + k S(u(t,x)),
\end{equation}
with $\nu$ and $k$ given constants. Application of the method of lines to~\eqref{eq:generalPDE} on discrete cells $\Omega_{j}$ leads to the coupled system of ODEs
\begin{equation} \label{eq:semidiscrete}
    \frac{\mathrm{d}}{\mathrm{d}t} \overline{U}_j(t) = - \frac{1}{\Delta x} \left[ \mathcal{F}_{j+\nicefrac12}(t) - \mathcal{F}_{j-\nicefrac12}(t) \right] + \frac{\nu^2}{2} K_j(t) + k S_j(t),
\end{equation}
where $\overline{U}_{j}(t)$
% $$
% \overline{U}_{j}(t) \approx \frac{1}{\Delta x} \int_{\Omega_{j}} u(t,x) \mathrm{d}x
% $$
is the approximation of the cell average of the exact solution in the cell $\Omega_{j}$ at time $t$.

Here, $\mathcal{F}_{j+\nicefrac12}(t)$ approximates $F(u(t,x_{j+\nicefrac12}),t,x_{j+\nicefrac12})$ with suitable accuracy and is computed as a function of the boundary extrapolated data, i.e.
$$
\mathcal{F}_{j+\nicefrac12}(t) = \mathcal{F}(U_{j+\nicefrac12}^{+}(t),U_{j+\nicefrac12}^{-}(t))
$$
and $\mathcal{F}$ is a consistent and monotone numerical flux, evaluated on two estimates of the solution at the cell interface, i.e $U_{j+\nicefrac12}^{\pm}(t)$. We focus on the class of central schemes, in particular we consider a local Lax-Friedrichs flux.
In order to construct a third-order scheme the values $U_{j+\nicefrac12}^{\pm}(t)$ at the cell boundaries are computed with the third-order CWENO reconstruction~\cite{CPSV:cweno,LPR:00:SIAMJSciComp}.

% The CWENO reconstruction procedure provides a third-order accurate reconstruction polynomial $\mathcal{R}_j(t,x)$, holding in each computational cell $\Omega_j$, which is computed from the given knowledge of the cell-averages, in such a way it satisfies essentially non-oscillatory properties. The polynomial $\mathcal{R}_j(t,x)$ is obtained as a convex combination of two one-sided linear interpolants and one centered interpolant parabola. In smooth regions the reconstruction polynomial guarantees the desired third-order accuracy, and automatically degrades to a second-order, one-sided, linear reconstruction in presence of discontinuities in the computational stencil. CWENO is very suitable for balance laws, as~\eqref{eq:generalPDE}, since the reconstruction polynomial is uniformly accurate, i.e. it provides reconstructions at any point $x$ within the computational cell with the desired accuracy.

The term $K_j(t)$ is a high-order approximation to the diffusion term in~\eqref{eq:generalPDE}. In the examples below we use the explicit fourth-order central difference employed in~\cite{KurganovLevy} for convective-diffusion equations with a general dissipation flux, and which uses point-values reconstructions computed with the CWENO polynomial.

Finally, $\mathcal{S}_j(t)$ is the numerical source term which is typically approximated as
$\mathcal{S}_j(t) = \sum_{q=0}^{N_q} \omega_q S(\mathcal{R}_j(t,x_q))$,
where $x_q$ and $\omega_q$ are the nodes and weights of a quadrature formula on $\Omega_j$. We employ three point Gaussian quadrature formula matching the order of the scheme.

System~\eqref{eq:generalPDE} is finally solved by the classical third-order (strong stability preserving) SSP Runge-Kutta with three stages~\cite{JiangShu:96}. At each Runge-Kutta stage, the cell averages are used to compute the reconstructions via the CWENO procedure and the boundary extrapolated data are fed into the local Lax-Friedrichs numerical flux. The initial data are computed with the three point Gaussian quadrature. The time step $\Delta t$ is chosen in an adaptive fashion and all the simulations are run with a CFL of $0.45$.

\subsection{Validation of the Moment Model Analysis}

In Section~\ref{Moment}, thanks to the corresponding moment model, we have extensively discussed properties of the solution of the mean-field limit of the SimResNet using a linear activation function. In this subsection we provide a numerical evidence of the theoretical findings by considering the following Gaussian probability distribution
$$
	g_0(x) = \frac{1}{\sqrt{2\pi}} \exp\left\{ -\frac{(x-1)^2}{2} \right\}
$$
as initial condition of~\eqref{eq:meanfield1D}. This represents the distribution of the given input data in a classical viewpoint of a neural network. Namely, we assume that the input measurements are distributed as a Gaussian.

\begin{figure}[t!]
\centering
\includegraphics[width=1\textwidth]{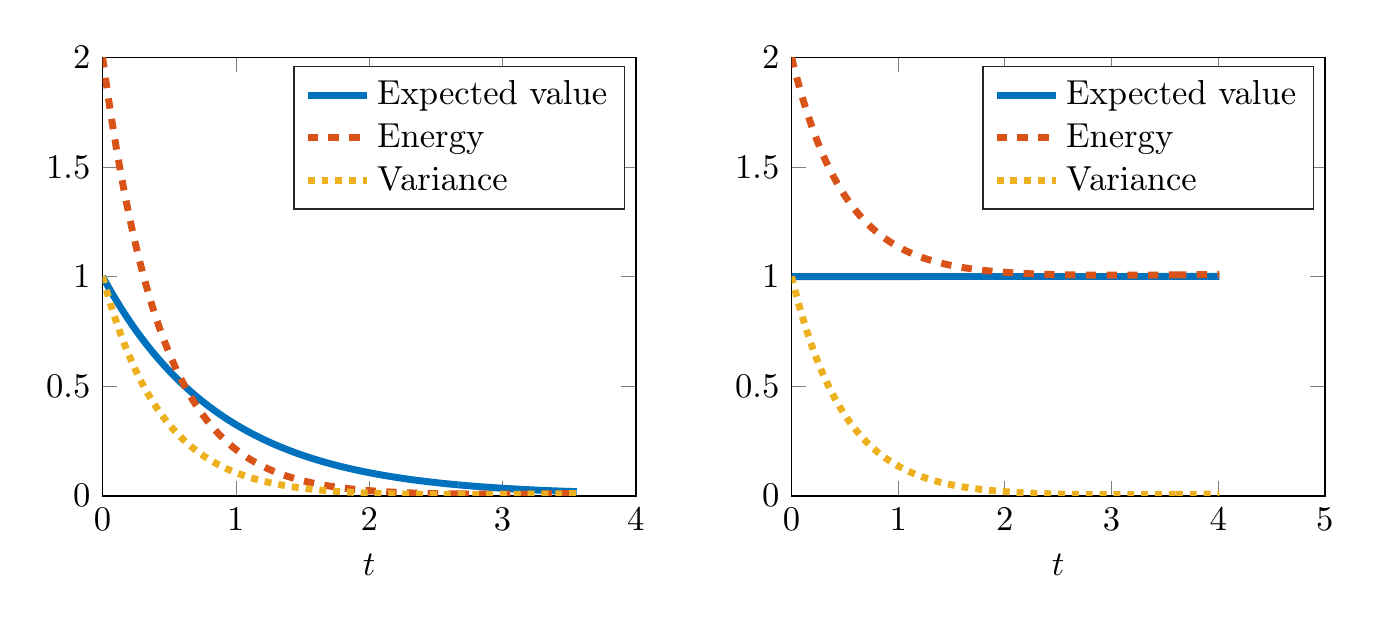}
\caption{Left: Moments of our PDE model with $\sigma(x)=x, w=-1, b=0$.
Right:  Moments of our PDE model with $\sigma(x)=x, w=-1, b=-\frac{m_1(t)}{m_0(0)}$.
}\label{clustering}
\end{figure}

\begin{figure}[t!]
\centering
\includegraphics[width=1\textwidth]{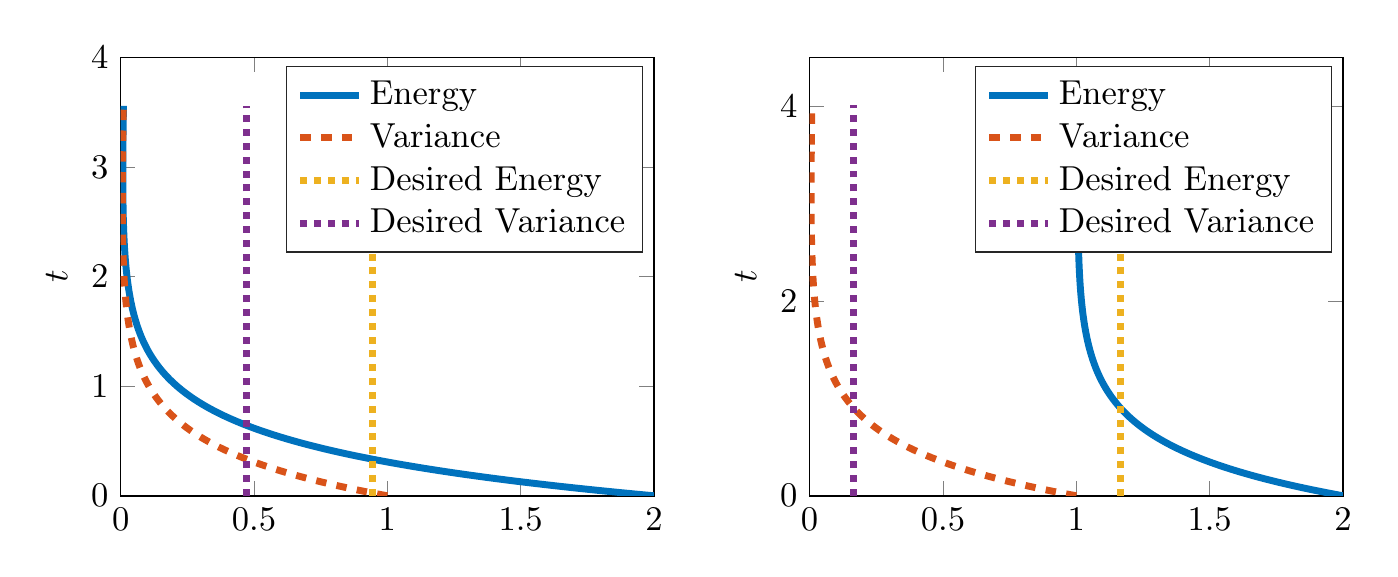}
\caption{Left: The energy and variance plotted against the desired values with $\sigma(x)=x, w=-1, b=0$.  Right: The energy and variance plotted against the desired values with $\sigma(x)=x, w=-1,  b=-\frac{m_1(t)}{m_0(0)}$.  }\label{desired}
\end{figure}

As predicted by Proposition~\ref{Prop::Steady} a steady state of the mean-field equation is a combination of Dirac delta. The number of these delta functions depends on the choice of the activation function. Proposition~\ref{prop1} and Proposition~\ref{prop2}, instead, provide conditions on the parameters to obtain a Dirac delta as output of the network when we consider $\sigma=\sigma_I$. In particular, with the choice $w=-1$ and $b=0$ a decay to zero of the energy, of the first moment and of the variance is expected. This situation is depicted in the left plot of Figure~\ref{clustering}. Instead, choosing $b=-\frac{m_1(t)}{m_0(0)}$ guarantees conservation of the first moment as shown in the right plot of Figure~\ref{clustering}.

Finally, Figure~\ref{desired} illustrates the findings of Corollary~\ref{cor2}. We compute the time needed in order to reach a fixed desired threshold of the values of the energy and of the variance. This means that, from a mean field view-point of a neural network, we can gain insights on the final time we consider in order to meet the variance of the target distribution, thus avoiding over-fitting phenomena. 

\subsection{Machine Learning Applications}

% We first aim to discuss the probabilistic interpretation which we have performed due to our kinetic description.
% In general it is possible to translate classical machine learning applications to such an probabilistic viewpoint.
% In the following
We present here the experiments based on the mean-field and on the kinetic approach, i.e.~relying on the statistical interpretation of the neural network process. We consider problems of classification and regression, typical of machine learning applications. In all these simulations the activation function is chosen to be the hyperbolic tangent, thus $\sigma(x)=\sigma_T(x)=\tanh(x)$.

\subsubsection{A Classification Problem} \label{MLApplications}
Consider a classification problem as follows. We measure a quantity (e.g.~the length of a vehicle) and we need to identify, or classify, the type of the object related to that measurement (e.g.~car or truck). Therefore, in a classification problem, the task of the neural network is to determine the type given a measurement. 
Keeping in mind the example of classifying vehicles from their length measurement, an experimental data set might look like Table~\ref{InputExample} below, with a suitable scalar label for the classifiers.

\begin{table}[!ht]
\caption{Example of a data set for a classification problem.}\label{InputExample}
\centering
\begin{tabular}[c]{|c||c|c|c|c|c|c|c}
\hline
Measurement & 3 & 3.5 & 5.5 & 7 & 4.5 & 8 & \dots \\
\hline
Classifier & car & car & truck & truck & car & truck & \dots \\
\hline
\end{tabular}
\end{table}

\begin{figure}[t!]
	\centering
		\includegraphics[width=1\textwidth]{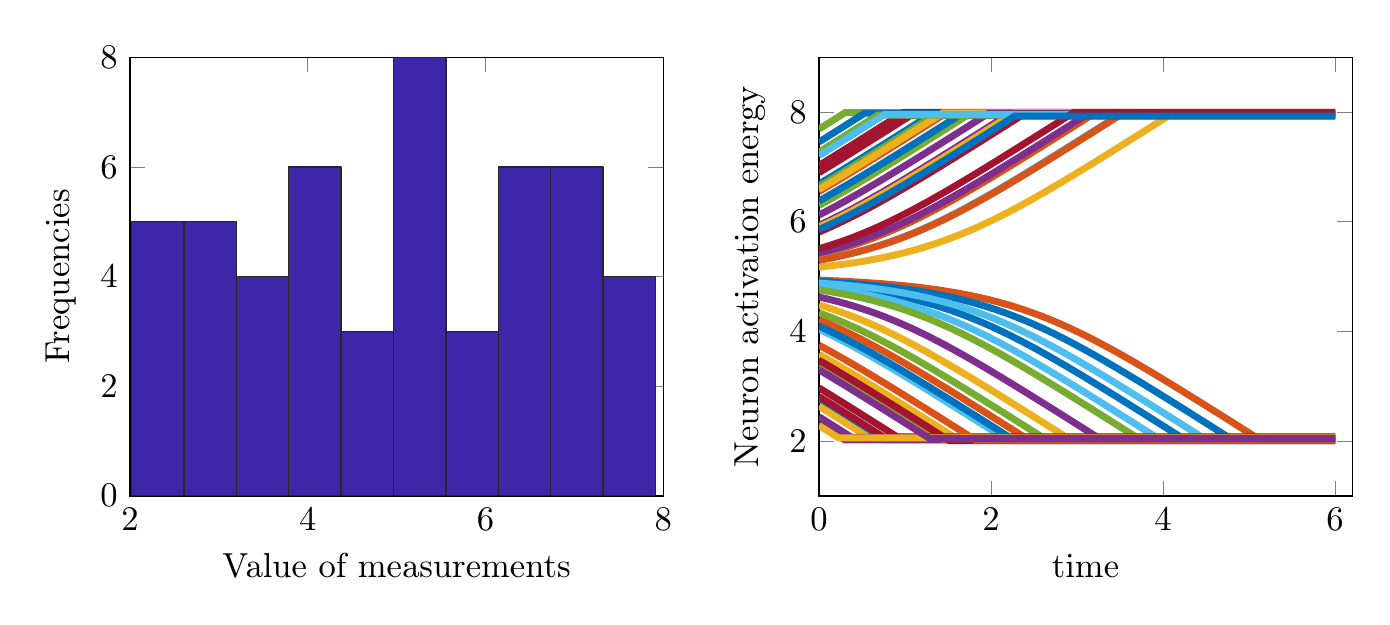}
	\caption{We consider $50$ vehicles with measured length $2$ and $8$ obtained as uniformly distributed random realizations.
		Left:  Histogram of the measured length of the vehicles. Right: Trajectories of the neuron activation energies of the $50$ measurements.
	}\label{HistRegression}
\end{figure}

\begin{figure}[t!]
	\centering
		\includegraphics[width=1\textwidth]{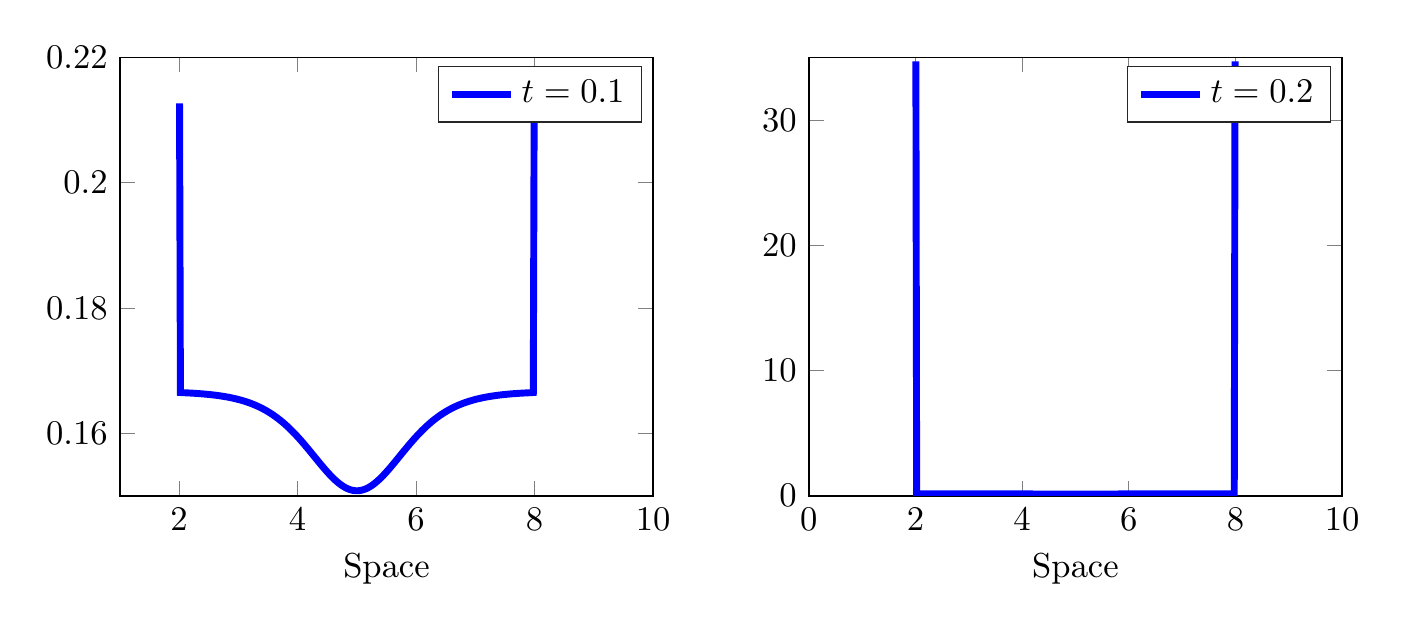}
	\caption{Solution of the mean field neural network model at different time steps. The initial value is a uniform distribution on $[2,8]$ and the weight and bias is chosen as $w=1,\ b=-5$. 
	}\label{RegressionSol}
\end{figure}

A statistical interpretation of the input data set can be obtained by considering the normalized histogram of the frequencies of the given measurements, as in the left plot of Figure~\ref{HistRegression}. Then, a continuous approximation of the histogram of the input measurements is used as initial distribution of the mean-field limit of the neural network. In this example the classifier is artificially associated to the binary values $2$, for cars, and $8$, for trucks. Thus, the output of the mean-field equation is expected to be two Dirac delta distributions located at the chosen binary values distinguishing the classifier. According to Proposition~\ref{Prop::Steady}, in order to have two Dirac delta at steady state we need to choose an activation function with at least two zeros in the computational domain. Another possibility is to use an activation function with only one zero, requiring the introduction of zero flux boundary conditions on the numerical scheme.

On the particle level, cf.~right plot of Figure~\ref{HistRegression}, we observe the convergence in time of given measurements to the expected clusters. In this case a simple evolution of the time continuous microscopic model~\eqref{eq:simplifiedNN} is employed.
%For the evolution of the mean field equation, this type of experiment requires the introduction of zero flux boundary conditions on the numerical scheme.
The solution of the mean-field formulation of the neural network model is depicted in Figure~\ref{RegressionSol} at different times.
Certainly, as already discussed, the mean-field equation provides a good approximation of the microscopic dynamics when the number of input measurements is very large. However, due to the simplistic formulation of the classification problem we observe a good agreement of the results obtained at the discrete and at the continuous level.

\subsubsection{A Regression Problem}

We may have given measurements at fixed locations. These measurements might be disturbed possibly due to measurement errors as in the left plot of Figure~\ref{RegressionProb}.
In a regression problem the task of the neural network is to find a linear fit $y=mx+q$ of those data points, thus learning its slope $m$ and the vertical intercept $q$, i.e.~the intersection of the linear function with the $y$ axis. This type of problem can be reformulated in a two-dimensional setting, where the given measurements are used to generate an input data set of slopes and vertical intercepts, see the center and the right plots of Figure~\ref{RegressionProb}.

\begin{figure}[t!]
	\centering
	\includegraphics[width=0.32\textwidth]{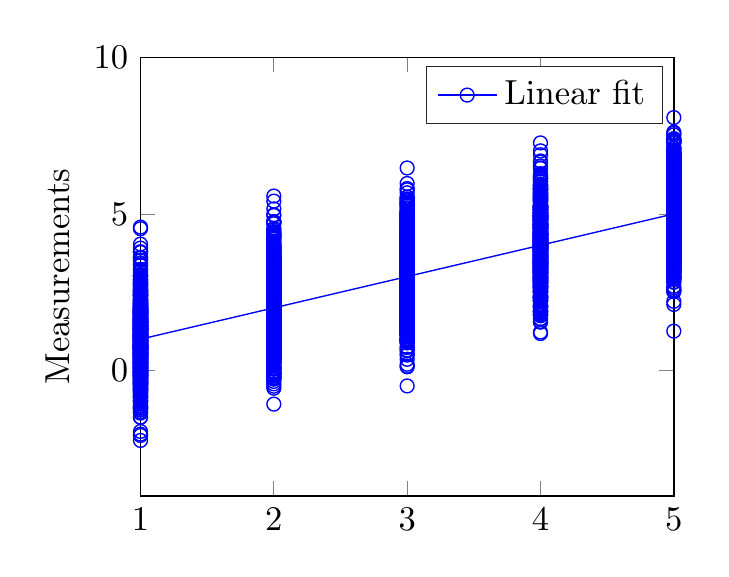}
	\includegraphics[width=0.32\textwidth]{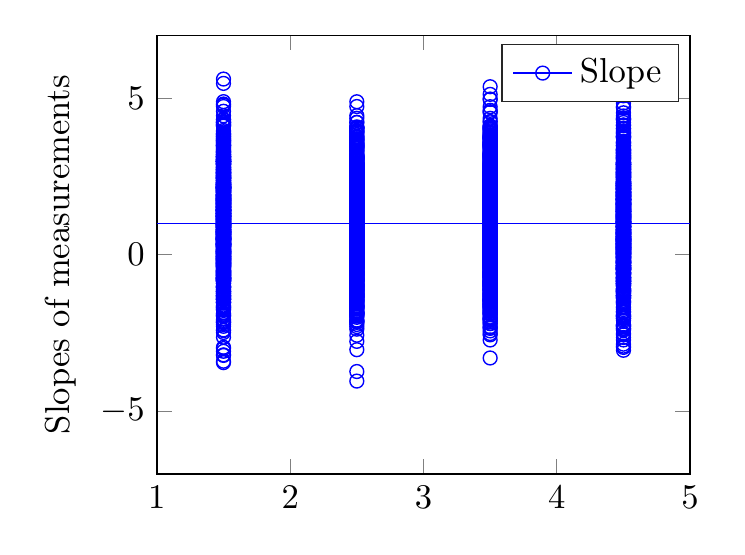}
	\includegraphics[width=0.32\textwidth]{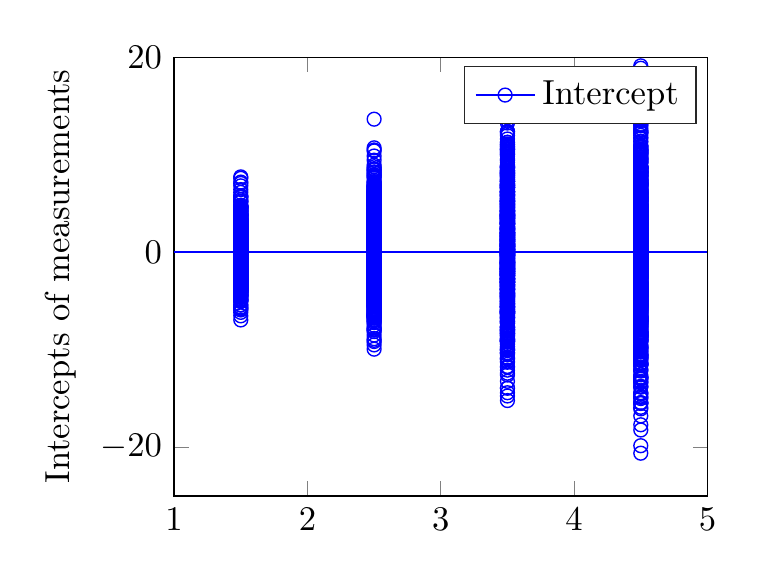}
	\caption{Left: Regression problem with $5\cdot10^3$ measurements at fixed positions around $y=x$. Measurement errors are distributed according to a standard Gaussian. Center: Numerical slopes computed out of the previous measurements. Right: Numerical intercepts computed out of the previous measurements.
	}\label{RegressionProb}
\end{figure}

\begin{figure}[t!]
	\centering
		\includegraphics[width=0.32\textwidth]{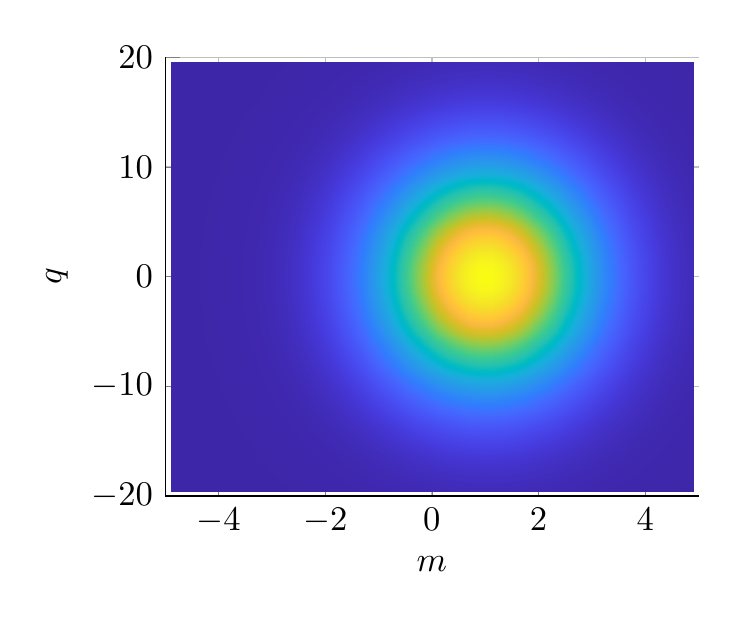}
		\includegraphics[width=0.32\textwidth]{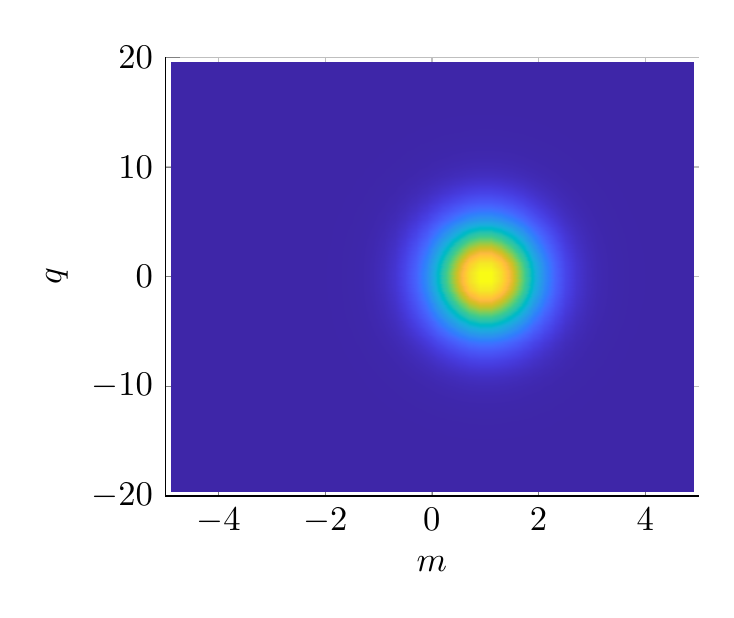}
		\includegraphics[width=0.32\textwidth]{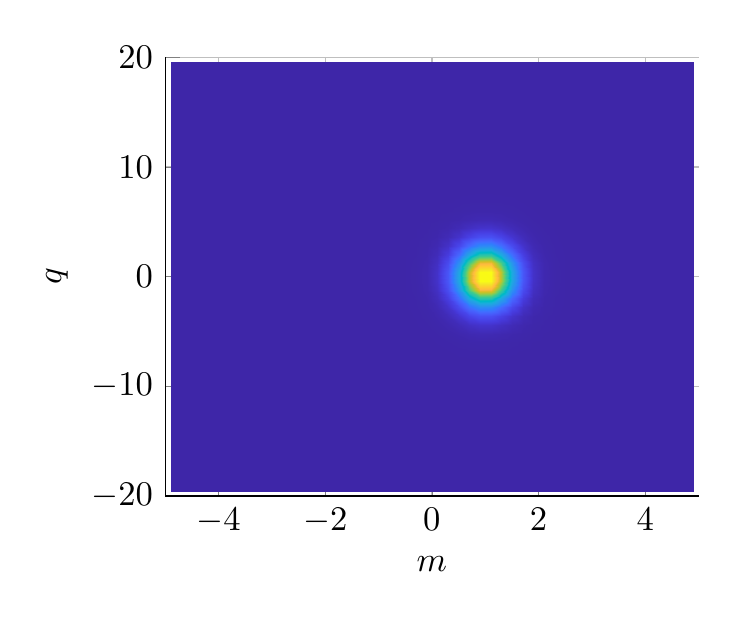}
	\caption{Evolution at time $t=0$ (left plot), $t=1$ (center plot), $t=2$ (right plot) of the mean field neural network model~\eqref{eq:2dmeanfield} for the regression problem with weights $w_{xx}=1$, $w_{xy}=w_{yx}=0$, $w_{yy}=-1$, and bias $b_x=-1$, $b_y=0$. 
	}\label{Regression}
\end{figure}

Also in this application we can build a statistical interpretation of the data set by considering the histogram of input data and its continuous approximation which is used as initial condition of the mean-field limit equation~\eqref{eq:meanfield}. We point-out that in this example we need to consider the two dimensional version of the mean-field PDE~\eqref{eq:meanfield}, namely
\begin{equation} \label{eq:2dmeanfield}
	\partial_t g(t,x,y) + \partial_x \left( \frac{w_{xx} x + w_{xy} y}{2} + b_x \right) + \partial_x \left( \frac{w_{yx} x + w_{yy} y}{2} + b_y \right) = 0.
\end{equation}
The solution of the mean field equation collapses to a Dirac delta distribution located at $(m,q)=(1,0)$, as we expect to recover from the regression problem. 
See Figure~\ref{Regression}. The parameters used to solve~\eqref{eq:2dmeanfield} are specified in the caption of Figure~\ref{Regression}. Further, we employ the identity activation function and we again stress the fact that in this choice, as well as in the choice of the parameters, we are using the theoretical findings of the moment model analysis in Section~\ref{Moment}.

\subsection{Experiment on the Forward Training of Weights and Bias}

We aim to present the benefit of the sensitivity analysis in Section~\ref{sensitivity} which has lead to a forward algorithm for the updates of the weights and the bias. To this end, we consider again a regression problem, with the same idea presented in the previous section. However, for simplicity and without loss of generality, we restrict the regression example to the one dimensional setting in which the task is learning the slope of the linear fit.
	
\begin{figure}[t!]
	\centering
	\includegraphics[width=\textwidth]{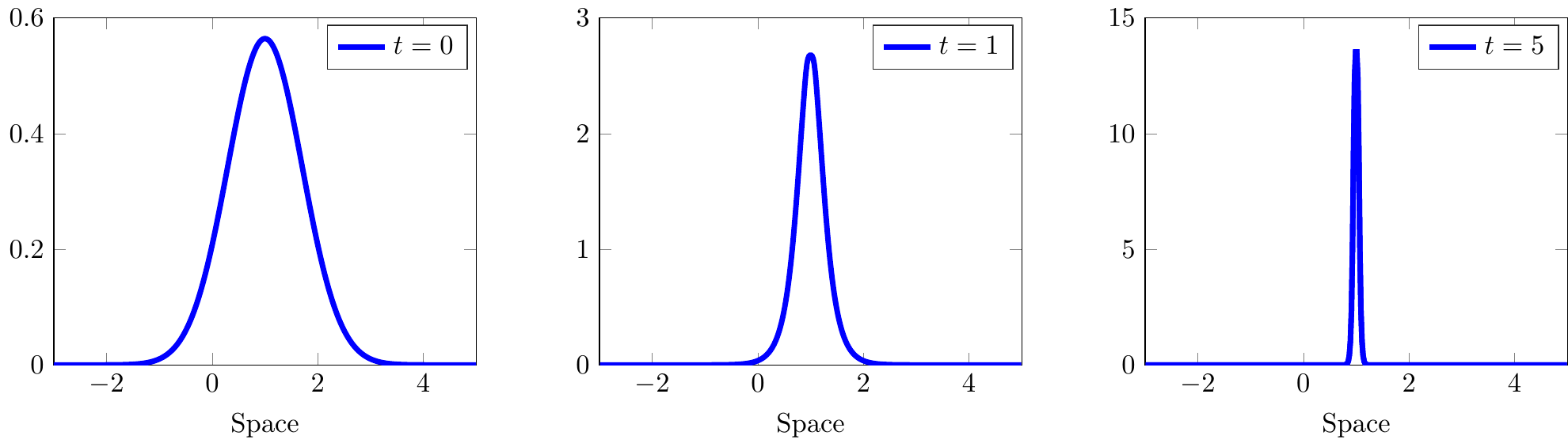}
	\caption{Evolution at time $t=0$ (left plot), $t=1$ (center plot), $t=5$ (right plot) of the one dimensional mean field neural network model for the regression problem with weight $w=1$ and bias $b=-1$. 
	}\label{Regression1D}
\end{figure}	

The distribution of the input data is assumed to be a Gaussian with mean and variance equal to one, and the target distribution is a Dirac delta centered at location $x=1$. The parameters of the mean field PDE, i.e.~the weights and bias, are $w=1$ and $b=-1$, respectively, as in the previous section. Therefore, the evolution of the one dimensional mean field equation can be observed in Figure~\ref{Regression} along the $m$ direction, and depicted in Figure~\ref{Regression1D} for the sake of readability.

\begin{figure}[t!]
    \centering
		\includegraphics[width=1\textwidth]{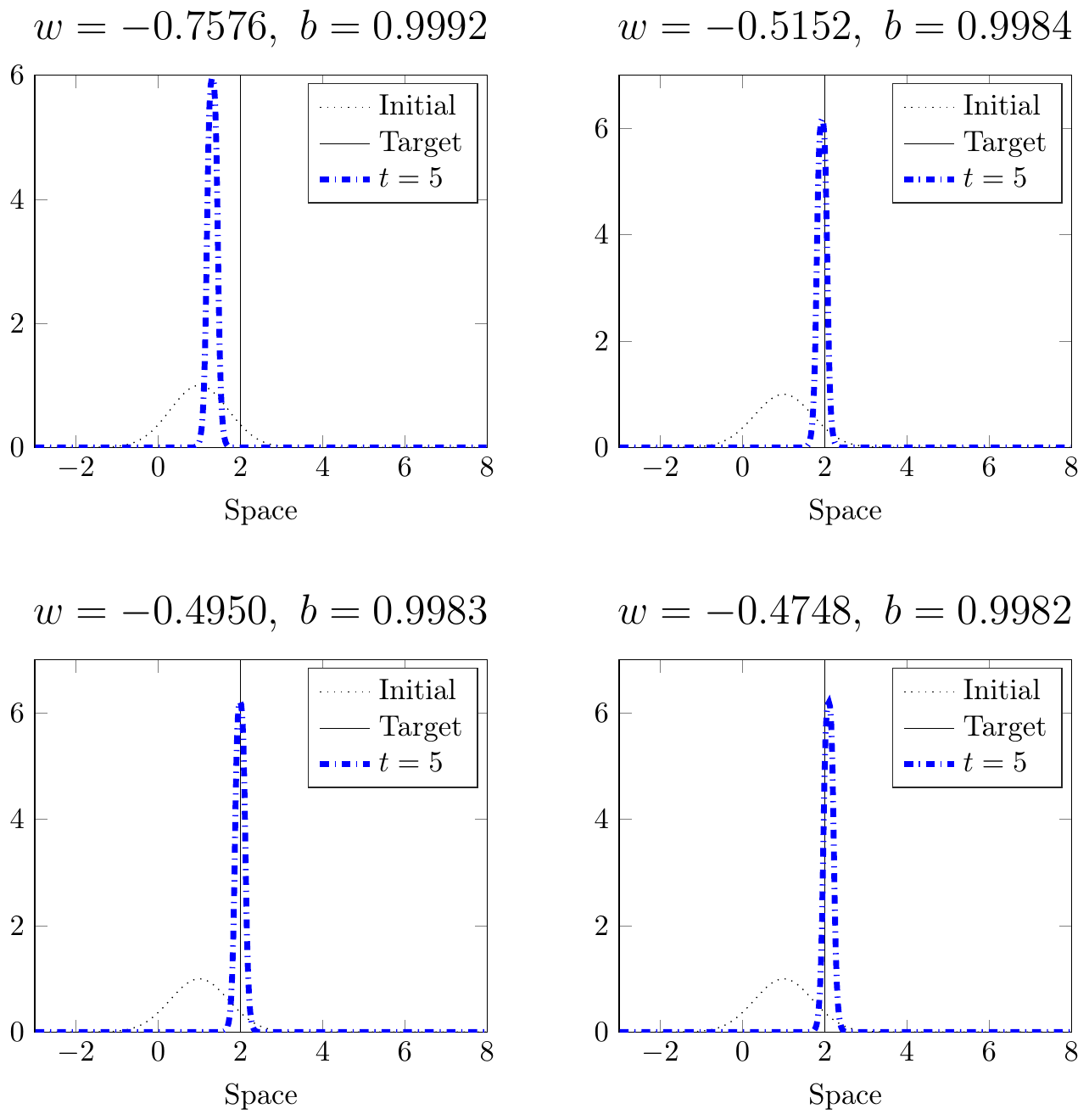}
	\caption{Results of the mean field neural network model with updated weights and bias in the case of a novel target. 
	}\label{GradUp}
\end{figure}

As Figure~\ref{Regression} shows, the solution at time $t=5$, i.e.~$g(t=5,x)$, is close to the target which is a Dirac delta centered at $x=1$. In order to investigate the application of the forward algorithm for the update of weight and bias, we consider a new target distribution, namely a Dirac delta centered at $x=2$. In a classical situation, we would need to re-train the network to compute new optimal parameters. Here, we use Algorithm~\ref{alg:sensitivity}, arising from the adjoint calculus in Section~\ref{sensitivity}, with a fixed step size $\gamma=2$ in order to update the weight and the bias. In Figure~\ref{GradUp} we show the solution of the mean field neural network model at $t=5$ for different number of gradient steps. At each of these steps we get a weight and bias, which is indicated in the title of each plot. After three step of the forward algorithm we obtain new parameters which are able to drive the solution of the mean field equation towards the new target, even starting from the same input data. This example shows how the forward algorithm can be used in order to update weights and bias in case of perturbations in the initial input distribution or in the target distribution.

\subsection{The Fokker-Planck Equation}

In this section we aim to validate the results of Section~\ref{FPsteadystate} on a typical example. In particular, the goal is to show that it is possible to fit a standard normal Gaussian distribution, which is taken as target distribution, under the conditions on parameters, activation and diffusion function stated in~Section~\ref{FPsteadystate}. Here, we consider uniform distribution on $[-1,-\frac12]$ as initial condition of the Fokker-Planck model~\eqref{eq:fokkerplanck} and evolve it in time. As presented in Section~\ref{FPsteadystate} the Fokker-Planck type interpretation of the neural network is able to fit the Gaussian distribution in the steady state if
we choose the identity as activation function, any $w^\infty<0$, $b^\infty=0$ and $K(x)=1$. This approach allows us to drive any initial input to the given target under this choices.

Recall that, on the contrary, and as proven in Section~\ref{Moment}, the mean field neural network can perform in the case of a hyperbolic tangent or identity activation function only clustering tasks if the conditions $w^\infty<0$, $b^\infty=0$ are verified.
This means that for large times the distribution approaches a Dirac delta distribution and consequently it is not possible to fit a Gaussian distributed target by using the deterministic SimResNet model. In this perspective, stochastic neural network with stochastic output layers, resulting in a Fokker-Planck kinetic interpretation, allows to fit more general target distributions under the same conditions.

% More precisely we obtain as steady states Dirac delta functions. One possibility in order to obtain alternative steady states is to introduce noise in the microscopic 
% SimResNet dynamics. This leads to Fokker-Planck type model with a variety of different steady states.

%\begin{figure}[t!]
%	\centering
%		\includegraphics[width=1\textwidth]{initial-target-FP.pdf}
%	\caption{Left: Histogram of 100 measurements uniformly distributed between $[-1,\frac12]$.	Right: Histogram of the standard Gaussian distributed target values.  
%	}\label{HistFP}
%\end{figure}

\begin{figure}[t!]
	\centering
		\includegraphics[width=1\textwidth]{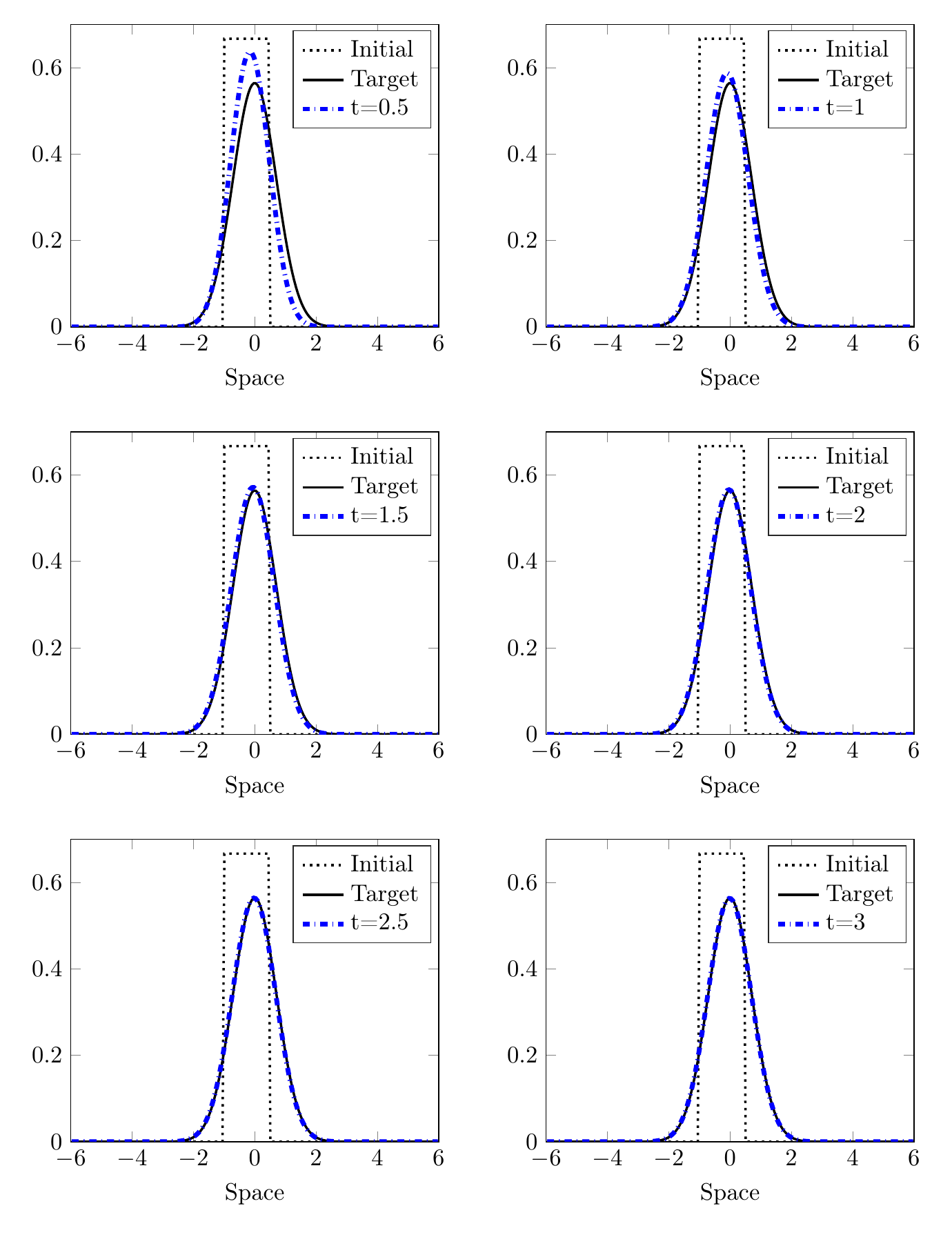}
	\caption{Solution of the Fokker-Planck neural network model at different times. Here, we have chosen the identity as activation function with weight $w=-1$, bias $b=0$ and diffusion function $K(x)=1$. }\label{SolFP}
\end{figure}

% In our example we aim to show that starting with a uniform distribution
% it is possible to reach a Gaussian target, therefore avoiding clustering. The uniform distribution is given on $[-1,-\frac12]$ and we consider a standard normal Gaussian distribution as target. 

The solution of the Fokker-Planck neural network model for different time steps is presented in Figure \ref{SolFP} showing the convergence to the given target. %Observe that, since we want to obtain a Gaussian steady state, we need to choose the identity activation function as given by the steady state characterization in Section~\ref{FPsteadystate}.

%\clearpage

\section{Conclusion}

Starting from the classical formulation of a residual neural network, we have considered a simplified structure of it, which consists in the assumption that each layer has the same number of neurons. This number is fixed by the size of the input signal. The effectiveness of a network with this structure has been demonstrated in previous works~\cite{GebhardtTrimborn,Bobzin2021}. Moreover, assuming this structure underlies the derivation of neural differential equations. In fact, by interpreting the discrete structure provided by the layers as a fictitious time discretization, we have derived the corresponding time continuous limit which leads to a differential formulation of the process of a network. This, in turn, has lead to the computation of the mean-field limit in the number of measurements. We have switched from a microscopic perspective on the level of time and input data to a statistical interpretation provided by the mean-field limit.

The resulting mean-field equation has been analyzed in terms of existence of solutions and characterization of steady states. Furthermore, under some assumptions, as considering an identity activation function, we have analyzed moment model properties. We have investigated the sensitivity of the loss function with respect to the weights and the bias bias which has allowed to determine a forward algorithm for the parameter update when a change of the input and of the target distribution is introduced, avoiding re-training of the network via back-propagation algorithms.
Finally, we have derived a Boltzmann description of the simplified residual neural network and extended it to the case of a noisy setting, motivated by stochastic residual neural networks with stochastic layers. As consequence, non trivial steady states have been obtained for the limiting Fokker-Planck type model. In the last section we have validated our analysis and have presented simple machine learning applications, based on regression and classification problems.

%Our study may yield insights in order to understand the performance of residual neural networks. E.g.~the moment analysis gives practical estimates on the simulation time and how to choose bias and weight. The gradient algorithm derived form the sensitivity analysis provides us with an update formula to recompute bias and weight: after changes in initial or target conditions. %Probably most interestingly, we have seen that the mean field neural network model has in special situations only a  Dirac delta function as unique steady state. In these cases the mean field neural network performs clustering tasks. In comparison to the mean field model, the Fokker-Planck neural network model is able to exhibit non-trivial steady states.

% Several extension of the derived kinetic approach are possible and under investigation. First or all we aim to add an optimization of the weights and bias to our kinetic model. This enables us to perform a fair comparison of classical residual neural networks with backpropagation with our approach. Secondly we aim to extend our sensitivity analysis in order to analyze a change in the input or the target distribution. 

%\clearpage

\subsection*{Acknowledgments}
The authors thank the Deutsche Forschungsgemeinschaft (DFG, German Research Foundation) for the financial support through 20021702/GRK2326, 333849990/IRTG-2379,
HE5386/15,18-1,19-1,22-1,23-1 and under Germany’s Excellence Strategy EXC-2023 Internet
of Production 390621612. The funding through HIDSS-004 is acknowledged.

M. Herty and T. Trimborn acknowledge the support by the ERS Prep Fund - Simulation and Data Science. The work was partially funded by the Excellence Initiative of the German federal and state governments.

%\clearpage
%
%\appendix
%
%\section{Appendix}
%
%\begin{definition}
%We have an energy bound of the function $g(t,y)\ t\geq 0,\ y\in\R$ at time $t$ iff 
%$$
%m_2(0)> m_2(t),
%$$
%holds.
%\end{definition}
%
%
%\begin{definition}
%We have energy decay of the function $g(t,y)\ t\geq 0,\ y\in\R$ if for any $t_1<t_2$
%$$
%m_2(t_1)\geq  m_2(t_2),
%$$
%holds. This means that the energy is decreasing with respect to time. 
%\end{definition}
%
%
%\begin{definition}
%We have concentration of the function $g(t,y)\ t\geq 0,\ y\in\R$ at time $t$ iff 
%$$
%\mathbb{V}(g(0,y))> \mathbb{V}(g(t,y))),
%$$
%holds. Here $\mathbb{V}$ denotes the variance. 
%\end{definition}
%
%
%\begin{definition}
%We have aggregation of the function $g(t,y)\ t\geq 0,\ y\in\R$ if for any $t_1<t_2$
%$$
%\mathbb{V}(g(t_1,y))\geq  \mathbb{V}(g(t_2,y))),
%$$
%holds. This means that the variance  is decreasing with respect to time. 
%\end{definition}

%-- LITERATUR ----------------------------------------------------------%
	%\clearpage
	\bibliography{literaturmean.bib}
		\bibliographystyle{abbrv}

\end{document}